\documentclass[11pt,reqno]{amsart}

\usepackage[T1]{fontenc}
\usepackage{amsmath,amssymb,amsthm,amsfonts,mathtools}
\usepackage{enumitem}
\usepackage{hyperref}
\usepackage{doi}
\hypersetup{colorlinks=true,linkcolor=blue,citecolor=blue,urlcolor=blue}

\newtheorem{thm}{Theorem}[section]
\newtheorem{lem}[thm]{Lemma}
\newtheorem{prop}[thm]{Proposition}
\newtheorem{cor}[thm]{Corollary}

\newtheorem{conj}[thm]{Conjecture}
\theoremstyle{definition}

\newtheorem{rem}[thm]{Remark}

\numberwithin{equation}{section}

\newcommand{\K}{\mathbb{K}}
\newcommand{\R}{\mathbb{R}}
\newcommand{\C}{\mathbb{C}}
\newcommand{\PP}{\mathbb{P}}
\newcommand{\EE}{\mathbb{E}}
\newcommand{\abs}[1]{\lvert #1\rvert}
\newcommand{\norm}[1]{\lVert #1\rVert}
\DeclareMathOperator{\diag}{diag}
\DeclareMathOperator{\rhoo}{\rho}
\DeclareMathOperator{\sgn}{sgn}

\begin{document}
	
	\title[Rump's 100 Euro conjecture]{Plank theorems, Gaussian probabilistic estimates and Rump's 100 Euro conjecture}
	
	\author[T.~Zhang]{Teng Zhang}
	\address{School of Mathematics and Statistics, Xi'an Jiaotong University, Xi'an 710049, P. R. China}
	\email{teng.zhang@stu.xjtu.edu.cn}
	
	
	\begin{abstract}
		We prove Rump's $100$ Euro conjecture by deriving a weighted affine escape theorem from Ball's plank theorem in [Invent. Math. \textbf{104} (1991)]. More precisely, let $\K\in\{\R,\C\}$ and let $A\in\K^{n\times n}$. For every $1\le p\le \infty$, we obtain an $\ell_p$-escape principle controlled by the row $\ell_q$-norms of $A$. Its cube case shows that $\abs{A}e=ne$ with all-one vector $e$ implies the existence of a nonzero vector $x$ satisfying $\norm{x}_{\infty}\le 1$ and $\abs{Ax}\ge e\ge |x|$, and hence settles the conjecture. As a consequence, we prove the global comparison
		$
		\rhoo(\abs{A})\le n\,\rho_{\K}(A),
		$
		where $\rho_{\K}$ is sign-real (complex) spectral radius.
		This is the sharp form of Rump's Perron--Frobenius type estimate without the factor $3+2\sqrt2$. Moreover, our $\ell_\infty$-escape principle  sharpens Rump's remarkable result in [SIAM Rev. \textbf{41} (1999)], namely the relation between the
		entrywise distance to singularity of a matrix and its entrywise (Bauer--Skeel) condition number. Finally, the weaker Euclidean row condition  is also investigated, including sharp quantitative bounds and counterexamples to possible strengthenings. In particular, we use Gaussian probabilistic estimates to establish a complex analogue of a conjecture by Bünger, the latter being stronger than the 100 Euro conjecture.
	\end{abstract}
	
	\maketitle
	\tableofcontents
	
	\section{Introduction}
	
	The Perron--Frobenius theory of nonnegative matrices is one of the cornerstones of matrix analysis and has numerous important applications \cite{HJ12,Mac00}. We begin with a brief account of its historical development.
	
	\subsection{Early history}
	
	In 1907, Perron \cite{Per07} proved his celebrated theorem, now known as \emph{Perron's theorem}:
	
	\medskip
	\emph{if $A$ is a positive matrix, then its spectral radius $\rho(A)$ is a positive and simple eigenvalue of $A$, it admits a positive eigenvector, and every other eigenvalue $\lambda$ satisfies $|\lambda|<\rho(A)$.}
	
	\medskip
	In 1912, Frobenius \cite{Fro12} extended this result from positive matrices to nonnegative matrices:
	
	\medskip
	\emph{if $A$ is an irreducible and nonnegative matrix, then $\rho(A)$ is a positive and simple eigenvalue of $A$, it admits a positive eigenvector, and every nonnegative eigenvector corresponding to $\rho(A)$ is a scalar multiple of that positive eigenvector.}
	\medskip
	
	The Collatz--Wielandt principle \cite{Col42,Wie50} expresses the Perron root through max-min and min-max characterizations and turns spectral information into an extremal problem:
	
	\medskip
	\emph{if $A$ is an irreducible and nonnegative matrix of order $n$ and $x$ is a positive vector, then}
	\[
	\rho(A)
	=
	\max_{x>0}\min_{1\le i\le n}\frac{(Ax)_i}{x_i}
	=
	\min_{x>0}\max_{1\le i\le n}\frac{(Ax)_i}{x_i}.
	\]
	The Krein--Rutman theorem \cite{KR48} extends the theory to compact cone-preserving operators on Banach spaces, thereby placing Perron--Frobenius theory within the general framework of positive operator theory. In a different direction, nonlinear Perron--Frobenius theory shows that several key features of the classical theory persist for order-preserving homogeneous maps on cones; see, for example, \cite{LN12}. These developments make clear that Perron--Frobenius theory is not merely a collection of results about nonnegative matrices, but a robust paradigm for extracting spectral information from positivity or order structure. Perron--Frobenius theory also provides deep structural information: algebraic simplicity in the primitive case, cyclic peripheral spectrum in the imprimitive case, monotonicity, and far-reaching connections with convexity, graph theory, probability, and dynamical systems. Standard modern accounts may be found in \cite{BP94,HJ12, Sen06}.
	
	\subsection{Modern developments}
	
	A call for ``challenges in matrix theory'' in 1996 \cite{Cha96} raised the following problem:
	
	\begin{center}
		\emph{Can we extend the Perron--Frobenius theory beyond the class of nonnegative matrices?}
	\end{center}
	
	In response, Rump \cite{Rum97b} initiated in 1997 a Perron--Frobenius type theory for matrices without sign restrictions. His starting point was the introduction of the sign-real spectral radius, a quantity which, although defined for arbitrary real matrices, retains several structural features of the Perron root of a nonnegative matrix. This theory was later refined through the study of cycle products \cite{Rum98}, variational and Collatz--Wielandt type characterizations \cite{Rum02}, the extension to complex matrices \cite{Rum03a}, and applications to interval regularity and $P$-matrices \cite{Roh12,Rum03b}. More recently, further developments were given by B\"unger \cite{Bun17}, B\"unger and Seeger \cite{BS24}, and Radons and Tonelli-Cueto \cite{RT23}.
	
	Before giving a detailed account of Rump's generalization of the Perron--Frobenius theory, we introduce some notation. We write $\mathbb{R}_{>0}^{\,n\times n}$ and $\mathbb{R}_{\ge 0}^{\,n\times n}$ for the sets of positive and nonnegative real $n\times n$ matrices, respectively. Similarly, $\mathbb{R}_{>0}^n$ and $\mathbb{R}_{\ge 0}^n$ denote the sets of positive and nonnegative vectors in $\mathbb{R}^n$. Correspondingly, for $x\in\mathbb{R}^n$, the notation $x>0$ means that $x_i>0$ for all $i$, and likewise for $x\ge 0$. For a matrix $A=(a_{ij})\in\C^{m\times n}$ and a vector $x=(x_i)\in \mathbb{C}^n$, absolute values and inequalities are understood entrywise. We write
	\[
	\abs{A}:=(\abs{a_{ij}}),\qquad \abs{x}:=(\abs{x_i}),\qquad e:=(1,\dots,1)^T.
	\]
	If $A\in\C^{n\times n}$, then $A^*$ denotes its conjugate transpose; for $\mu\subseteq\{1,\dots,n\}$, $A[\mu]$ denotes the principal submatrix of $A$ indexed by $\mu$.
	For a row or column vector $a=(a_1,\dots,a_n)$ with all $a_j\in \C$ and $1\le p\le\infty$, we use the standard notation
	\[
	\norm{a}_p:=
	\begin{cases}
		\left(\sum_{j=1}^n \abs{a_j}^p\right)^{1/p}, & 1\le p<\infty,\\[0.4em]
		\max\limits_{1\le j\le n}\abs{a_j}, & p=\infty.
	\end{cases}
	\]
We use
	$\operatorname{diag}(d_1,\dots,d_m)
$
to	denote the diagonal matrix whose diagonal entries are $d_1,\dots,d_m$.
	
	Next, we return to the main topic.	For $\mathbb{K}\in\{\mathbb{R}, \mathbb{C}\}$ and a matrix $A\in\mathbb{K}^{n\times n}$, Rump defines
	\begin{equation}\label{eq:intro-rhoK}
		\rho_{\mathbb K}(A):=
		\max\{|\lambda|:\lambda\in\mathbb{K},\ \exists\,0\neq x\in\mathbb{K}^n \text{ such that } |Ax|=|\lambda x|\}.
	\end{equation}
	Clearly, when $\mathbb{K}=\mathbb{R}_{> 0}$ ($\mathbb{R}_{\ge 0}$), this is precisely the classical Perron root of a positive (nonnegative) matrix. For $\mathbb{K}=\mathbb{R}$, the quantity \eqref{eq:intro-rhoK} is the \emph{sign-real spectral radius} introduced in \cite{Rum97b}; for $\mathbb{K}=\mathbb{C}$, it is the \emph{sign-complex spectral radius} introduced in \cite{Rum03a}.
	
	The sign-real spectral radius originally arose in the study of the componentwise distance to the nearest singular matrix \cite{Rum97a,Rum97c}. One of the reasons for its importance is that the classical Collatz--Wielandt characterization extends to both the sign-real \cite[Theorem~3.1]{Rum97b} and sign-complex settings \cite[Theorem~2.4]{Rum03a}:
	\begin{equation}\label{eq:intro-cw}
		\rho_{\mathbb K}(A)=
		\max_{\substack{x\in\mathbb{K}^n\\ x\neq 0}}
		\min_{x_i\neq 0}
		\left|\frac{(Ax)_i}{x_i}\right|.
	\end{equation}
	In addition, $\rho_{\mathbb K}$ enjoys a number of basic Perron--Frobenius type properties. Among them, the following seven will be particularly relevant for us; see \cite[Lemma~2.1, Corollary~2.4]{Rum97b} for $\mathbb K=\R$ and \cite[Lemma~2.1, Theorem~2.7]{Rum03a} for $\mathbb K=\C$:
	\begin{align*}
		\text{(i)}\quad & \rho_{\mathbb K}(A^*)=\rho_{\mathbb K}(A),\\[0.3em]
		\text{(ii)}\quad & |S_1|=|S_2|=I \;\Longrightarrow\; \rho_{\mathbb K}(S_1AS_2)=\rho_{\mathbb K}(A),\\[0.3em]
		\text{(iii)}\quad & P\in\mathbb{R}^{n\times n}\ \text{permutation matrix}
		\;\Longrightarrow\; \rho_{\mathbb K}(P^TAP)=\rho_{\mathbb K}(A),\\[0.3em]
		\text{(iv)}\quad & D\in\mathbb{K}^{n\times n}\ \text{nonsingular diagonal}
		\;\Longrightarrow\; \rho_{\mathbb K}(D^{-1}AD)=\rho_{\mathbb K}(A),\\[0.3em]
		\text{(v)}\quad & \alpha\in\mathbb{K}
		\;\Longrightarrow\; \rho_{\mathbb K}(\alpha A)=|\alpha|\,\rho_{\mathbb K}(A),\\[0.3em]
		\text{(vi)}\quad & \mu\subseteq\{1,\dots,n\}
		\;\Longrightarrow\; \rho_{\mathbb K}(A[\mu])\le \rho_{\mathbb K}(A),\\[0.3em]
		\text{(vii)}\quad & A=(a_{ij})\ \text{triangular}
		\;\Longrightarrow\; \rho_{\mathbb K}(A)=\max_i |a_{ii}|.
	\end{align*}
	These properties, together with the max-min formula \eqref{eq:intro-cw}, show that $\rho_{\mathbb K}$ behaves in many respects like a generalized Perron root; see \cite{Rum97b,Rum02,Rum03a}. Moreover, $\rho_{\K}$ is continuous; see \cite{Rum97b} for $\K=\R$ and \cite{Rum03a} for $\K=\C$.	
Zalar \cite{Zal99} proved that the linear preservers of the sign-real spectral radius are generated by four basic transformations and their combinations.
	
	\subsection{Rump's 100 Euro conjecture}
	A fundamental comparison established by Rump is
	\begin{equation}\label{eq:rump-factor}
		\rho_{\mathbb K}(A)\le \rhoo(\abs{A})\le (3+2\sqrt2)\,n\,\rho_{\mathbb K}(A)
		\qquad (\mathbb K\in\{\R,\C\}),
	\end{equation}
	see \cite[Theorem~5.7]{Rum97b} for $\mathbb K=\R$ and \cite[Theorem~6.3]{Rum03a} for $\mathbb K=\C$. This inequality links the generalized Perron root $\rho_{\mathbb K}(A)$ to the classical Perron root of the entrywise absolute value $\abs{A}$, and it is one of the main motivations for the problems studied in the present paper.
	
	The right-hand inequality in \eqref{eq:rump-factor} naturally raises the question whether the factor $3+2\sqrt2$ is really necessary. The corresponding sharp form for $\mathbb K=\R$ was formulated by Rump as the \emph{$100$ Euro conjecture} \cite{Rum100}:
	\begin{equation}\label{eq:100}
		\rhoo(\abs{A})\le n\,\rho_{\mathbb R}(A).
	\end{equation}
	It can be traced back to his 1997 work on sign-real spectral radii \cite[Section~6]{Rum97b} and was later highlighted in his problem note offering a prize for its solution \cite{Rum100}. With a history spanning nearly three decades, this conjecture has become one of the central problems in the theory of sign-real spectral radii and sign-real expansive matrices. It has also entered the problem literature of the International Linear Algebra Society; see, for example, Problem~63-1 in the IMAGE Problem Corner \cite{BMW23}. For a recent account of the area and the role of this conjecture within it, see \cite{BS24}.
	
	A convenient normalized formulation equivalent to \eqref{eq:100} is the following.
	
	\begin{conj}[Rump]\label{conj:100}
		Let $A\in\R^{n\times n}$ satisfy $\abs{A}e=ne$. Then there exists $x\in\R^n\setminus\{0\}$ such that
		\[
		\abs{Ax}\ge \abs{x}.
		\]
	\end{conj}
	
	Conjecture~\ref{conj:100} is sharp in the following dimension-free sense: there is no constant $\alpha>1$, independent of $n$, such that for every $n$ and every matrix $A\in\mathbb R^{n\times n}$ with $|A|e=ne$, there exists a nonzero vector $x\in\mathbb R^n$ satisfying
	\[
	|Ax|\ge \alpha |x|.
	\]
	Indeed, let
	\begin{align}\label{eq:Sn}
		S_n=(s_{ij})_{i,j=1}^n,\qquad
		s_{ij}=
		\begin{cases}
			0,& i=j,\\
			1,& i<j,\\
			-1,& i>j,
		\end{cases}
	\end{align}
	
	and define
	\[
	A_n:=\frac{n}{n-1}S_n.
	\]
	Then $|A_n|e=ne$. Moreover, Rump \cite[Lemma~5.6]{Rum97b} proved that
	\[
	\rho_{\mathbb R}(S_n)=1
	\]
	and therefore
	\[
	\rho_{\mathbb R}(A_n)=\frac{n}{n-1}.
	\]
	
	Now suppose that there existed a constant $\alpha>1$, independent of $n$, such that for every matrix $A$ with $|A|e=ne$ one could find $x\neq 0$ satisfying
	\[
	|Ax|\ge \alpha |x|.
	\]
	Applying this to $A_n$, and using Rump's Collatz--Wielandt type characterization \eqref{eq:intro-cw}, we would obtain
	\[
	\rho_{\mathbb R}(A_n)\ge \alpha.
	\]
	Hence
	\[
	\frac{n}{n-1}=\rho_{\mathbb R}(A_n)\ge \alpha
	\qquad\text{for all }n.
	\]
	This is impossible, since $n/(n-1)\to 1$ as $n\to\infty$ and $\alpha>1$ is fixed. Therefore no dimension-independent constant $\alpha>1$ can replace the factor $1$ in Conjecture~\ref{conj:100}.
	
	Geometrically, the condition $\abs{A}e=ne$ in Conjecture~\ref{conj:100} means that each row of $A$ lies on the boundary of the $\ell_1$-ball of radius $n$, equivalently, on the $n$-dimensional cross-polytope with vertices $\pm n e_1,\dots,\pm n e_n$. In the nonnegative setting, Perron--Frobenius theory shows that this normalization forces the Perron root of $\abs{A}$ to be $n$, with Perron eigenvector $e$. Conjecture~\ref{conj:100} asks whether, under the same row-wise $\ell_1$ normalization but without any sign restriction on $A$, one can still find a nonzero vector $x$ such that $\abs{Ax}\ge \abs{x}$. In this sense, the conjecture may be regarded as a sign-independent extension of a Perron--Frobenius type existence statement.
	
	From another viewpoint, the inequality $\abs{Ax}\ge \abs{x}$ is a coordinatewise non-contraction condition, so the problem may also be interpreted in the Banach lattice $\ell_\infty^n$ as the search for a direction along which the operator $A$ does not decrease magnitude. The geometric picture above is particularly well suited to our approach, since it turns the conjecture into an escape problem from coordinate slabs inside the cube.
	
	A stronger Euclidean variant of Conjecture~\ref{conj:100} was proposed by Bünger about 25 years ago \cite{Rum100}. It is known as the \emph{$200$ Euro conjecture}\footnote{In fact, this is yet another 100 euro conjecture; see \cite{Rum100}. A proof would automatically settle Conjecture~\ref{conj:100} as well, thus raising its effective value to 200 euros. Consequently, if the first conjecture were proved, Conjecture~\ref{conj:200} would still have 100 euros left on the table.}, as it supersedes Conjecture~\ref{conj:100}; see \cite{Rum100} and also \cite[Conjecture~27]{BS24}. It is particularly appealing because it is stronger than Conjecture~\ref{conj:100} for all \(n \ge 2\). It reads as follows.

%	A stronger Euclidean variant of Conjecture~\ref{conj:100}, due to B\"unger and also known as the \emph{$200$ Euro conjecture}\footnote{In fact, this is yet another 100 euro conjecture; see \cite{Rum100}. A proof would automatically settle Conjecture~\ref{conj:100} as well, thus raising its effective value to 200 euros. Consequently, if the first conjecture were proved, Conjecture~\ref{conj:200} would still have 100 euros left on the table.}, was proposed about twenty-five years ago \cite{Rum26}; see \cite{Rum100} and also \cite[Conjecture~27]{BS24}. It states the following.
	
	\begin{conj}[B\"unger]\label{conj:200}
		Let $A=(a_{ij})\in\R^{n\times n}$ with $n\ge 2$, and let $r_i=(a_{i1},\dots,a_{in})$ denote the $i$-th row of $A$. Assume that
		\[
		\norm{r_i}_2\ge \sqrt{n-1}
		\qquad (i=1,\dots,n).
		\]
		Then there exists $x\in\R^n\setminus\{0\}$ such that
		\[
		\abs{Ax}\ge \abs{x}.
		\]
	\end{conj}
	
As has been noted by B\"unger \cite{Rum100},  the assumption in Conjecture~\ref{conj:200} is weaker than the normalization condition in Conjecture~\ref{conj:100}. Indeed, if $\abs{A}e=ne$, then every row $r_i$ satisfies $\norm{r_i}_1=n$, and hence, by the Cauchy--Schwarz inequality,
	\[
	\norm{r_i}_2\ge \frac{\norm{r_i}_1}{\sqrt n}=\sqrt n>\sqrt{n-1}
	\qquad (i=1,\dots,n).
	\]
	Conjecture~\ref{conj:200} is  sharp in the sense that the conclusion $\abs{Ax}\ge \abs{x}$ cannot be replaced by $\abs{Ax}>\abs{x}$. Consider the matrix $A=S_n$ defined in \eqref{eq:Sn}. Each row of $A$ has Euclidean norm $\sqrt{n-1}$, so $A$ satisfies the hypothesis of Conjecture~\ref{conj:200}. Moreover, by \cite[Lemma~5.6]{Rum97b},
	\[
	\rho_{\mathbb R}(A)=1.
	\]
	If every matrix satisfying the Euclidean row condition admitted a nonzero vector $x$ with
	\[
	\abs{Ax}>\abs{x},
	\]
	then, by Rump's Collatz--Wielandt type characterization \eqref{eq:intro-cw}, we would obtain
	\[
	\rho_{\mathbb R}(A)>1,
	\]
	a contradiction. On the other hand, for
	\[
	x=(1,-1,0,\dots,0)^T
	\]
	one has
	\[
	Ax=(-1,-1,0,\dots,0)^T,
	\qquad\text{hence}\qquad
	\abs{Ax}=\abs{x}.
	\]
	Therefore the factor $1$ in Conjecture~\ref{conj:200} is best possible.
	\begin{rem}
		To obtain the sharp estimate in \eqref{eq:100}, it would suffice to prove Conjecture \ref{conj:100}. However, Conjecture \ref{conj:200} is particularly appealing because it assumes weaker conditions and is sharp.		
	\end{rem}
	\subsection{Main results}
	In this paper, we first provide an $\ell_p$-family of escape theorems by using Ball's plank theorem in \cite{Bal91}.
	
	\begin{thm}\label{thm:mainp}
		Let $1\le p\le \infty$, and let $q$ be the conjugate exponent of $p$, that is, $1/p+1/q=1$ with the usual convention $1/\infty=0$. Let $\K\in\{\R,\C\}$, let
		$A=(a_{ij})\in\K^{m\times n}$, and let $r_i=(a_{i1},\ldots,a_{in})$ denote the $i$-th row of $A$. Assume that
		\[
		\norm{r_i}_q\ge mt
		\qquad (i=1,\dots,m),
		\]
		for some $t>0$. Then there exists a nonzero vector $x\in\K^n$ with $\norm{x}_p\le 1$ such that
		\[
		\abs{Ax}\ge t\,e.
		\]
	\end{thm}
	
	We remark that 
	when $m=n$, the constant $n$ for the condition $\norm{r_i}_q\ge nt$ in Theorem~\ref{thm:mainp} is optimal for $p=1$ as by $A=I_n$. 
	For other values of $p$, however, the optimal constant is not clear. In particular, the extreme case $p=\infty$ is more subtle and is closely related to the extremal phenomena behind Conjecture~\ref{conj:100}.

	Taking $p=\infty$,  $\K=\R$, $m=n$ and $t=1$ in Theorem~\ref{thm:mainp}, we immediately obtain the cube case needed for  Conjecture~\ref{conj:100}.
	
	\begin{cor}\label{cor:cube}
		Let $A\in\R^{n\times n}$ satisfy $\abs{A}e\ge ne$. Then there exists a vector $x\in[-1,1]^n\setminus\{0\}$ such that
		\[
		\abs{Ax}\ge e\ge \abs{x}.
		\]
		In particular, Conjecture~\ref{conj:100} holds.
	\end{cor}
	\begin{rem}
		As we will see in Proposition~\ref{prop:weighted-distance-improvement}, the estimate $|Ax|\ge e$ is the key ingredient in sharpening Rump's bound relating the entrywise distance to the nearest singular matrix to the Bauer--Skeel condition number in \cite{Rum99}. In particular, it is stronger than the conclusion $|Ax|\ge |x|$.
	\end{rem}
	We next turn from the existence statement to its spectral consequence. Combining the escape theorem with Rump's Collatz--Wielandt type characterization of $\rho_{\K}$ \eqref{eq:intro-cw}, we will obtain the following global comparison. Since these are standard conclusions in matrix theory and to avoid distracting the reader, we defer the proof to the Appendix~\ref{sec:proof-compare}.
	
	\begin{thm}\label{thm:global}
		Let $\K\in\{\R,\C\}$ and $A\in\K^{n\times n}$. Then
		\[
		\rhoo(\abs{A})\le n\,\rho_{\K}(A).
		\]
	\end{thm}
	Our next theorem is derived from Ortega-Moreno's optimal plank theorem in Hilbert spaces \cite{Ort21}. The constant obtained in this way is in fact best possible: in general, there is no unit vector  $x\in\R^n$ such that
	$
	\abs{Ax}> 	\sqrt{n-1}\sin(\pi/(2n))e.
	$ As a corollary, we obtain a weaker form of Conjecture~\ref{conj:200} with the weaker coefficient
	$
	\sqrt{n-1}\sin(\pi/(2n)),
	$
	which is stronger than the coefficient $\sqrt{n-1}/n$ obtained by taking $m=n$ and $p=2$ in Theorem~\ref{thm:mainp}.
	\begin{thm}\label{thm:euclid-sine}
		Let $\K\in\{\R,\C\}$ and $A\in\K^{n\times n}$ with $n\ge 2$, and let $r_i$ denote the $i$-th row of $A$. Assume that
		\[
		\norm{r_i}_2\ge \sqrt{n-1}
		\qquad (i=1,\dots,n).
		\]
		Then there exists a unit vector $x\in\K^n$ with $\norm{x}_2=1$ such that
		\[
		\abs{Ax}\ge \sqrt{n-1}\sin\!\Bigl(\frac{\pi}{2n}\Bigr)e,
		\]
		where the constant $	\sqrt{n-1}\sin(\pi/(2n))$ is optimal when $\K=\R$. In particular,  there exists a nonzero vector $x\in\K^n$ with $\norm{x}_2= 1$ such that
		\[
		\abs{Ax}\ge \sqrt{n-1}\sin\!\Bigl(\frac{\pi}{2n}\Bigr)\,\abs{x}.
		\]
	\end{thm}
	\begin{rem}
		For any $1\le p\le \infty$, the stronger normalized conclusion $\abs{Ax}\ge e$ with $\norm{x}_p\le 1$ is false under the hypotheses of Conjecture~\ref{conj:200}. The related counterexamples are given in Section~\ref{sec:euclid}.
	\end{rem}
Next, we establish a complex analogue of Conjecture~\ref{conj:200} under weaker assumptions. The argument combines  Gaussian probabilistic estimates with a covering argument on punctured Euclidean space.
	\begin{thm}\label{thm:complex-general}
		Let $A=(a_{ij})\in \C^{n\times n}$ with $n\ge 2$, and let $r_i$ denote the $i$-th row of $A$. Assume that
		\begin{equation*}
			\norm{r_i}_2\ge 1
			\qquad (i=1,\ldots,n),\quad
			\text{ and }\quad
			\sum_{i=1}^n \frac{1}{1+\norm{r_i}_2^2}\le 1.
		\end{equation*}
		Then there exists a nonzero vector $x\in \C^n$ such that
		\[
		\abs{Ax}\ge \abs{x}.
		\]
	\end{thm}
Similarly, we establish the existence of a nonzero vector satisfying the conclusion of Conjecture~\ref{conj:200} under different assumptions. Moreover, when \(n=2\), these assumptions are equivalent to those in Conjecture~\ref{conj:200}. That is,  Conjecture~\ref{conj:200} is true for $n=2$.
	\begin{thm}\label{thm:real-general}
		Let $A=(a_{ij})\in \R^{n\times n}$ with $n\ge 2$, and let $r_i$ denote the $i$-th row of $A$. Assume that
		\begin{equation*}
			\norm{r_i}_2\ge 1
			\qquad (i=1,\ldots,n),\quad
			\text{ and }\quad
			\sum_{i=1}^n \frac{2}{\pi}\arctan\!\Bigl(\frac{1}{\norm{r_i}_2}\Bigr)\le 1.
		\end{equation*}
		Then there exists a nonzero vector $x\in \R^n$ such that
		\[
		\abs{Ax}\ge \abs{x}.
		\]
	\end{thm}
\begin{rem}\label{rem:general-sharpness}
	The conditions $\norm{r_i}_2\ge 1$ cannot be removed in either
	Theorem~\ref{thm:complex-general} or Theorem~\ref{thm:real-general}.
	Indeed, for any $M>1$, the matrix
	\[
	A_M=
	\begin{pmatrix}
		0 & M\\
		0 & M^{-1}
	\end{pmatrix}\in \R^{2\times 2}\subset \C^{2\times 2}
	\]
	satisfies
	\[
	\norm{r_1}_2=M,\qquad \norm{r_2}_2=M^{-1},
	\]
	and
	\[
	\frac{1}{1+\norm{r_1}_2^2}+\frac{1}{1+\norm{r_2}_2^2}=1.
	\]
	Moreover, since $\arctan(M)+\arctan(M^{-1})=\pi/2$, we also have
	\[
	\frac{2}{\pi}\arctan\!\Bigl(\frac{1}{\norm{r_1}_2}\Bigr)
	+
	\frac{2}{\pi}\arctan\!\Bigl(\frac{1}{\norm{r_2}_2}\Bigr)
	=1.
	\]
	However, there is no nonzero vector $x$ such that $\abs{A_Mx}\ge \abs{x}$.
	Indeed, if $x=(x_1,x_2)^T$ and $\abs{A_Mx}\ge \abs{x}$, then
	\[
	M^{-1}|x_2|=|(A_Mx)_2|\ge |x_2|,
	\]
	so $x_2=0$, and then
	\[
	0=|(A_Mx)_1|\ge |x_1|,
	\]
	hence $x_1=0$.
\end{rem}
Finally, in view of the numerical evidence, we propose the following real analogue of Theorem~\ref{thm:complex-general}. This conjecture may be regarded as a stronger variant of Conjecture~\ref{conj:200} when $n\ge 3$.
\begin{conj}
	Let $A=(a_{ij})\in \R^{n\times n}$ with $n\ge 2$, and let $r_i$ denote the $i$-th row of $A$. Assume that
	\begin{equation*}
		\norm{r_i}_2\ge 1
		\qquad (i=1,\ldots,n),\quad
		\text{ and }\quad
		\sum_{i=1}^n \frac{1}{1+\norm{r_i}_2^2}\le 1.
	\end{equation*}
	Then there exists a nonzero vector $x\in \R^n$ such that
	\[
	\abs{Ax}\ge \abs{x}.
	\]
\end{conj}
\medskip
\noindent\textbf{Organization of the paper.}
Section~\ref{sec:Ball} recalls a finite-dimensional form of Ball's plank theorem
and derives a weighted affine consequence in Banach spaces.
Section~\ref{sec:main} proves Theorem~\ref{thm:mainp}; from it we derive
Corollary~\ref{cor:cube}, while the proof of the global comparison
Theorem~\ref{thm:global} is deferred to Appendix~\ref{sec:proof-compare}.
Section~\ref{sec:distance} applies Theorem~\ref{thm:mainp} to the weighted
componentwise distance to singularity and obtains a sharp improvement of
Rump's estimate in \cite{Rum99}. Section~\ref{sec:BS24} records several
consequences for recent results of B\"unger and Seeger in \cite{BS24}.
Section~\ref{sec:euclid} studies the weaker Euclidean row condition, proving
Theorem~\ref{thm:euclid-sine} and giving related counterexamples.
Section~\ref{sec:gaussian} collects the Gaussian facts needed later, and
Section~\ref{sec:general-proof} proves Theorems~\ref{thm:complex-general}
and~\ref{thm:real-general}. Appendix~\ref{app:gaussian} contains routine
background on real and complex Gaussian variables. Finally,
Section~\ref{sec:further} contains further remarks on
Conjecture~\ref{conj:200}, including elementary planar results.

	\section{Ball's plank theorem and a weighted affine consequence}\label{sec:Ball}
	
	Let $X$ be a finite-dimensional real Banach space with norm $\norm{\cdot}_X$, and let $X^*$ denote its dual space, that is, the space of all continuous linear functionals on $X$, endowed with the dual norm
	\[
	\norm{f}_{X^*}:=\sup\{|f(x)|:\ x\in X,\ \norm{x}_X\le 1\}.
	\]
	We write
	\[
	B_X:=\{x\in X:\norm{x}_X\le 1\}
	\]
	for the closed unit ball of $X$.
	
	We use the following finite-dimensional form of Ball's plank theorem \cite[Theorem~2]{Bal91}.
	
	\begin{thm}[Ball]\label{thm:Ball}
		Let $X$ be a finite-dimensional real Banach space, and suppose that $\varphi_1,\dots,\varphi_m\in X^*$ satisfy $\norm{\varphi_i}_{X^*}=1$ for every $i$, that $b_1,\dots,b_m\in\R$, and that $w_1,\dots,w_m>0$ satisfy
		\[
		\sum_{i=1}^m w_i=1.
		\]
		Then there exists $x\in B_X$ such that
		\[
		\abs{\varphi_i(x)-b_i}\ge w_i
		\qquad (i=1,\dots,m).
		\]
	\end{thm}
	
	The following renormalized form will be used in the proof of Theorem~\ref{thm:mainp}.
	
	\begin{cor}\label{cor:affine}
		Let $X$ be a finite-dimensional real Banach space, and suppose that $f_1,\dots,f_m\in X^*\setminus\{0\}$, that $c_1,\dots,c_m\in\R$, and that $\lambda_1,\dots,\lambda_m>0$ satisfy
		\[
		\sum_{i=1}^m \frac{\lambda_i}{\norm{f_i}_{X^*}}\le 1.
		\]
		Then there exists $x\in B_X$ such that
		\[
		\abs{f_i(x)-c_i}\ge \lambda_i
		\qquad (i=1,\dots,m).
		\]
	\end{cor}
	
	\begin{proof}
		Set
		\[
		\varphi_i:=\frac{f_i}{\norm{f_i}_{X^*}},
		\qquad
		b_i:=\frac{c_i}{\norm{f_i}_{X^*}},
		\qquad
		w_i:=\frac{\lambda_i}{\norm{f_i}_{X^*}}
		\qquad (i=1,\dots,m).
		\]
		Then $\norm{\varphi_i}_{X^*}=1$ and $w_i>0$ for all $i$. If $\sum_{i=1}^m w_i=1$, the conclusion follows directly from Theorem~\ref{thm:Ball}.
		
		Assume now that $\sum_{i=1}^m w_i<1$. Since every $f_i$ is nonzero, the space $X$ is nontrivial. Choose $\psi\in X^*$ with $\norm{\psi}_{X^*}=1$ and define
		\[
		\varphi_{m+1}:=\psi,
		\qquad
		b_{m+1}:=0,
		\qquad
		w_{m+1}:=1-\sum_{i=1}^m w_i.
		\]
		Then $w_{m+1}>0$ and $\sum_{i=1}^{m+1}w_i=1$. Applying Theorem~\ref{thm:Ball} to $\varphi_1,\dots,\varphi_{m+1}$ yields $x\in B_X$ such that
		\[
		\abs{\varphi_i(x)-b_i}\ge w_i
		\qquad (i=1,\dots,m+1).
		\]
		For $i\le m$ this is exactly
		\[
		\abs{f_i(x)-c_i}\ge \lambda_i.
		\]
	\end{proof}
	
	\section{Proof of Theorem~\ref{thm:mainp}}\label{sec:main}
	
	In this section, we prove  Theorem~\ref{thm:mainp}  by using Corollary~\ref{cor:affine}.
	\begin{proof}[Proof of Theorem~\ref{thm:mainp}]
		Let $X=(\K^n,\norm{\cdot}_p)$, regarded as a real Banach space.
		For each $i=1,\dots,m$, define the $\K$-linear row functional
		\[
		f_i(z):=(Az)_i=\sum_{j=1}^n a_{ij}z_j
		\qquad (z\in\K^n),
		\]
		and let
		\[
		g_i(z):=\operatorname{Re} f_i(z)
		\qquad (z\in\K^n).
		\]
		Then $g_i$ is a real-linear functional on $X$.
		
		By the usual $\ell_p$--$\ell_q$ duality over $\K$, for each $i=1,\ldots, m$,
		\[
		\norm{f_i}=\norm{r_i}_q.
		\]
		We claim that also
		\[
		\norm{g_i}=\norm{r_i}_q.
		\]
		Indeed, for every $z\in\K^n$,
		\[
		|g_i(z)|=|\operatorname{Re} f_i(z)|\le |f_i(z)|\le \norm{f_i}\,\norm{z}_p,
		\]
		so $\norm{g_i}\le \norm{f_i}$.
		
		For the reverse inequality, let $z\in\K^n$ with $\norm{z}_p\le 1$ and $f_i(z)\neq 0$, and set
		\[
		\omega:=\frac{\overline{f_i(z)}}{|f_i(z)|}.
		\]
		Then $|\omega|=1$, hence $\norm{\omega z}_p=\norm{z}_p\le 1$, and
		\[
		g_i(\omega z)=\operatorname{Re}\bigl(f_i(\omega z)\bigr)
		=\operatorname{Re}\bigl(\omega f_i(z)\bigr)
		=|f_i(z)|.
		\]
		Therefore
		\[
		\norm{g_i}\ge |f_i(z)|.
		\]
		Taking the supremum over all $z$ with $\norm{z}_p\le 1$ yields
		\[
		\norm{g_i}\ge \norm{f_i}.
		\]
		Hence
		\[
		\norm{g_i}=\norm{f_i}=\norm{r_i}_q.
		\]
		
		Now set $\lambda_i:=t$ for all $i$. Since $\norm{r_i}_q\ge mt$, we have
		\[
		\sum_{i=1}^m \frac{\lambda_i}{\norm{g_i}}
		=
		\sum_{i=1}^m \frac{t}{\norm{r_i}_q}
		\le
		\sum_{i=1}^m \frac{t}{mt}
		=1.
		\]
		Applying Corollary~\ref{cor:affine} to the real Banach space $X$, with the real functionals $g_i$, with $c_i=0$, and with $\lambda_i=t$, we obtain a vector $x\in X=\K^n$ such that $\norm{x}_p\le 1$ and
		\[
		|g_i(x)|\ge t
		\qquad (i=1,\dots,m).
		\]
		Hence
		\[
		|(Ax)_i|
		=
		|f_i(x)|
		\ge
		|\operatorname{Re} f_i(x)|
		=
		|g_i(x)|
		\ge t
		\qquad (i=1,\dots,m),
		\]
		that is,
		\[
		|Ax|\ge t\,e.
		\]
		
		Moreover, $x\neq 0$, since otherwise $|Ax|=0$, contradicting $|Ax|\ge t\,e$ with $t>0$.
	\end{proof}

	\section{Weighted componentwise distance to singularity}\label{sec:distance}

	In this section, we explain how Theorem~\ref{thm:mainp} yields a sharp improvement of a classical result of Rump on the weighted componentwise distance to singularity \cite{Rum99}.
	
	The problem belongs to the broader framework of perturbation analysis for matrix inversion. In the classical normwise setting, if $A\in\R^{n\times n}$ is nonsingular and $\norm{\cdot}$ is a matrix norm subordinate to a vector norm, then the normwise distance from $A$ to the set of singular matrices is exactly the reciprocal of the condition number
	\[
	\kappa(A):=\norm{A}\,\norm{A^{-1}},
	\]
	namely,
	\[
	\min\bigl\{\alpha\ge0:\ \exists\,\widetilde A\in\R^{n\times n}\ \text{with }\norm{\widetilde A-A}\le \alpha \norm{A}
	\ \text{and}\ \widetilde A\ \text{singular}\bigr\}
	=
	\frac{1}{\kappa(A)};
	\]
	see, for example, \cite[Theorem~III.2.8]{SS90}. Thus, in the normwise theory, ``ill-conditioned'' and ``nearly singular'' are equivalent notions.
	
	For componentwise perturbations, however, the normwise condition number is no longer the most natural quantity, because it does not adequately reflect entrywise relative errors and is highly sensitive to diagonal scalings. This leads to the Bauer--Skeel approach to componentwise perturbation analysis \cite{Bau66,Ske79}; see also \cite{Dem92,Hig96,Roh89}. Given a nonsingular matrix $A\in\R^{n\times n}$ and a nonnegative weight matrix $E\in\R_{\ge0}^{n\times n}$, the quantity
	\[
	\||A^{-1}|E\|
	\]
	arises as the corresponding Bauer--Skeel type condition number for a fixed subordinate norm, while its scale-invariant refinement 	leads to the Bauer--Skeel condition number
	\[
	\operatorname{cond}(A,E):=\rho\!\bigl(|A^{-1}|E\bigr).
	\]
	This number is independent of diagonal row and column scalings and, in the important case of relative perturbations $E=|A|$, coincides with the minimum attainable $\ell_\infty$-condition number under two-sided diagonal scalings \cite{Bau63,Dem92,Rum99}. It is therefore natural to ask whether the reciprocal of $\rho(|A^{-1}|E)$ also controls the distance from $A$ to singularity when perturbations are measured componentwise.
	
	This question was formulated explicitly by Demmel for relative perturbations and was emphasized by Higham as a particularly natural conjecture \cite{Dem92,Hig96}. It was settled by Rump through a line of work on componentwise distance to singularity and sign-real spectral radii \cite{Rum97a,Rum97c,Rum99}. The resulting comparison is especially valuable because the exact computation of weighted componentwise distance is NP-hard in general \cite{PR93}, so one cannot expect a simple closed formula or an efficient exact algorithm.

	Let $A\in\R^{n\times n}$ be nonsingular, and let $E\in\R_{\ge0}^{n\times n}$ be a nonnegative weight matrix. Following Rump \cite{Rum99}, the weighted componentwise distance from $A$ to the set of singular matrices is defined by
	\[
	\sigma(A,E):=
	\min\bigl\{\alpha\ge0:\ \exists\,\widetilde E\in\R^{n\times n}\ \text{with}\ |\widetilde E|\le \alpha E
	\ \text{and}\ A+\widetilde E\ \text{singular}\bigr\},
	\]
	with the convention $\sigma(A,E):=\infty$ if no such $\alpha$ exists. The weight matrix $E$ prescribes the admissible entrywise size of the perturbation and thus provides a versatile way to model various entrywise perturbation patterns. For instance, the choice $E=|A|$ corresponds to componentwise relative perturbations, while $E=ee^T$ corresponds to entrywise absolute perturbations. More generally, specific matrix structures may be incorporated by setting selected entries of $E$ to zero, thereby forbidding perturbations in those positions and preserving the prescribed structure, such as tridiagonal structure and related sparsity patterns. A compactness argument shows that the minimum in the definition is attained whenever $\sigma(A,E)<\infty$ \cite{Rum99}.

	The quantity $\sigma(A,E)$ is the componentwise analogue of the normwise distance to singularity. In the framework of weighted componentwise perturbation analysis, one is led to compare it with
	\[
	\operatorname{cond}(A,E)=\rho\!\bigl(|A^{-1}|E\bigr).
	\]
	Rump answered this question in the affirmative by proving that
	\begin{equation}\label{eq:Rum99-bound}
		\frac{1}{\rho(|A^{-1}|E)}
		\le
		\sigma(A,E)
		<
		\frac{(3+2\sqrt2)n}{\rho(|A^{-1}|E)}.
	\end{equation}
	See \cite[eq.~(4) and Proposition~5.1]{Rum99}. The lower bound is classical and sharp, whereas the upper bound is the genuinely difficult part. Thus Rump showed that, even in the componentwise setting, ill-conditioning still forces proximity to singularity, up to a universal constant times the dimension.
	
	A central ingredient in Rump's proof is a characterization of $\sigma(A,E)$ in terms of the sign-real spectral radius. More precisely, writing $\rho_{\R}$ for the sign-real spectral radius, Rump proved
	\begin{equation}\label{eq:Rum99-characterization}
		\sigma(A,E)^{-1}
		=
		\max_{|\widehat E|\le E}\rho_{\R}(A^{-1}\widehat E),
	\end{equation}
	which is the form of his characterization most convenient for our purposes; see \cite[eq.~(9)]{Rum99}. This identity reduces the problem of bounding $\sigma(A,E)$ from above to the problem of finding lower bounds on $\rho_{\R}$ for suitable matrices. In \cite{Rum99}, those lower bounds were obtained from cycle-product estimates for the sign-real spectral radius; see \cite[Theorem~4.4 and Proposition~5.1]{Rum99}. In this way, the upper estimate in \eqref{eq:Rum99-bound} becomes an application of the Perron--Frobenius type theory for matrices without sign restrictions developed in \cite{Rum97b}. 
	
	For $\mu=\{i_1<\cdots<i_k\}\subseteq\{1,\dots,n\}$ and a matrix
	$M=(m_{ij})\in\R^{n\times n}$, we write
	\[
	M[\mu,:]
	:=
	\begin{pmatrix}
		m_{i_1 1} & \cdots & m_{i_1 n}\\
		\vdots & & \vdots\\
		m_{i_k 1} & \cdots & m_{i_k n}
	\end{pmatrix}
	\in \R^{k\times n},
	\]
	that is, $M[\mu,:]$ denotes the submatrix of $M$ formed by the rows indexed by $\mu$ and all columns. For a vector $z=(z_1,\dots,z_n)^T\in\R^n$, we write
	\[
	z_\mu:=(z_{i_1},\dots,z_{i_k})^T\in\R^k
	\]
	for the subvector of $z$ consisting of the entries indexed by $\mu$.
	
	Theorem~\ref{thm:mainp} allows one to revisit this argument from a different angle. Instead of using cycle products, we apply the rectangular $\ell_\infty$-escape theorem to a matrix naturally associated with a Perron vector of $|A^{-1}|E$. This leads to the following sharpening of Rump's upper bound.
	
	\begin{prop}\label{prop:weighted-distance-improvement}
		Let $A\in\R^{n\times n}$ be nonsingular, let $E\in\R_{\ge0}^{n\times n}$, and set
		\[
		B:=|A^{-1}|E,
		\qquad
		r:=\rho(B).
		\]
		Assume that $r>0$, and define
		\[
		k_*:=\min\Bigl\{|\operatorname{supp}z|:\ z\in\R_{\ge0}^n\setminus\{0\},\ Bz=rz\Bigr\},
		\]
		where
		$
		\operatorname{supp} z:=\{\,i\in\{1,\dots,n\}: z_i\neq 0\,\}
		$
		denotes the support of $z$. Then
		\[
		\sigma(A,E)
		\le
		\frac{k_*}{\rho(|A^{-1}|E)}
		\le
		\frac{n}{\rho(|A^{-1}|E)}.
		\]
		Moreover,  both inequalities are sharp.
	\end{prop}
	
	\begin{proof}
		Since $B$ is nonnegative and $r=\rho(B)>0$, the Perron--Frobenius theorem guarantees that the set
		\[
		\Bigl\{|\operatorname{supp}z|:\ z\in\R_{\ge0}^n\setminus\{0\},\ Bz=rz\Bigr\}
		\]
		is nonempty. Let $z\in\R_{\ge0}^n\setminus\{0\}$ satisfy
		\[
		Bz=rz
		\]
		and
		\[
		|\operatorname{supp}z|=k_*.
		\]
		Write
		\[
		\mu:=\operatorname{supp}z=\{i_1<\cdots<i_k\},
		\qquad
		k:=|\mu|=k_*,
		\qquad
		u:=Ez.
		\]
		Since $r>0$, we have $k\ge 1$. Moreover, $z_i>0$ for every $i\in\mu$.
		
		Set $M:=A^{-1}$ and define
		\[
		T:=D_z^{-1}M[\mu,:]D_u\in\R^{k\times n},
		\]
		where
		\[
		D_z:=\diag(z_{i_1},\dots,z_{i_k}),
		\qquad
		D_u:=\diag(u_1,\dots,u_n).
		\]
		Since $z_i>0$ for all $i\in\mu$, the diagonal matrix $D_z$ is invertible.
		
		For each $\ell=1,\dots,k$, let $i_\ell$ be the corresponding element of $\mu$. Then the $\ell$-th row of $T$ satisfies
		\[
		\|T_{\ell\bullet}\|_1
		=
		\frac1{z_{i_\ell}}\sum_{j=1}^n |m_{i_\ell j}|u_j
		=
		\frac{(|M|u)_{i_\ell}}{z_{i_\ell}}
		=
		\frac{(Bz)_{i_\ell}}{z_{i_\ell}}
		=
		r.
		\]
		Therefore Theorem~\ref{thm:mainp}, applied to the matrix $T$ with $p=\infty$, $m=k$, and
		\[
		t=\frac{r}{k},
		\]
		yields a vector $x\in\R^n\setminus\{0\}$ such that
		\[
		\|x\|_\infty\le 1
		\qquad\text{and}\qquad
		|Tx|\ge \frac{r}{k}e.
		\]
		
		Now define
		\[
		\widetilde E:=D_xE,
		\qquad
		D_x:=\diag(x_1,\dots,x_n).
		\]
		Since $\|x\|_\infty\le 1$, we have $|D_x|\le I$, and hence
		\[
		|\widetilde E|\le E.
		\]
		Moreover,
		\[
		\widetilde E z
		=
		D_xEz
		=
		D_xu
		=
		D_u x.
		\]
		Hence
		\[
		|D_z^{-1}M[\mu,:]\widetilde E z|
		=
		|Tx|
		\ge
		\frac{r}{k}e,
		\]
		and therefore
		\[
		|M[\mu,:]\widetilde E z|
		\ge
		\frac{r}{k}z_\mu.
		\]
		Equivalently,
		\[
		|(M\widetilde E z)_i|
		\ge
		\frac{r}{k}z_i
		\qquad (i\in\mu).
		\]
		If $i\notin\mu$, then $z_i=0$, so the same inequality holds trivially. Consequently,
		\[
		|A^{-1}\widetilde E z|
		=
		|M\widetilde E z|
		\ge
		\frac{r}{k}z.
		\]
		
		By Rump's Collatz--Wielandt type characterization of the sign-real spectral radius \eqref{eq:intro-cw},
		\[
		\rho_{\R}(A^{-1}\widetilde E)\ge \frac{r}{k}.
		\]
		On the other hand, Rump's characterization \eqref{eq:Rum99-characterization} implies
		\[
		\sigma(A,E)^{-1}
		=
		\max_{|\widehat E|\le E}\rho_{\R}(A^{-1}\widehat E)
		\ge
		\rho_{\R}(A^{-1}\widetilde E)
		\ge
		\frac{r}{k}.
		\]
		Equivalently,
		\[
		\sigma(A,E)\le \frac{k}{r}
		=
		\frac{k_*}{\rho(|A^{-1}|E)}.
		\]
		The weaker estimate with $n$ follows from $k_*\le n$.
		
		It remains to discuss sharpness. Rump \cite[eqs.~(25)--(26)]{Rum99} exhibited matrices for which
		\[
		\sigma(A,|A|)=\frac{n}{\rho(|A^{-1}||A|)}.
		\]
 For these examples one has $k_*=n$, so equality in the second estimate yields equality in the first as well. Hence both inequalities are sharp.
	\end{proof}	
\begin{cor}\label{cor:weighted-distance-sharp}
	Let $A\in\R^{n\times n}$ be nonsingular and let $E\in\R_{\ge0}^{n\times n}$. Then
	\[
	\frac{1}{\rho(|A^{-1}|E)}
	\le
	\sigma(A,E)
	\le
	\frac{n}{\rho(|A^{-1}|E)},
	\]
	with the usual convention $1/0:=\infty$. In particular, \eqref{eq:Rum99-bound} holds with $n$ in place of $(3+2\sqrt2)n$. Moreover, both inequalities are sharp.
\end{cor}
\begin{proof}
	If $\rho(|A^{-1}|E)>0$, the upper bound follows from Proposition~\ref{prop:weighted-distance-improvement}, and the lower bound is the classical lower estimate in \eqref{eq:Rum99-bound}. If $\rho(|A^{-1}|E)=0$, then $|A^{-1}|E=0$. Since every column of $|A^{-1}|$ is nonzero, this forces $E=0$, and hence $\sigma(A,E)=\infty$. The sharpness statements follow from Proposition~\ref{prop:weighted-distance-improvement}.
\end{proof}
	\begin{rem}
	If one only wants the weaker estimate
	\[
	\sigma(A,E)\le \frac{n}{\rho(|A^{-1}|E)},
	\]
	then the proof in Proposition~\ref{prop:weighted-distance-improvement} can be simplified. Indeed, for $\varepsilon>0$, set
	\[
	E_\varepsilon:=E+\varepsilon ee^T.
	\]
	Since $A^{-1}$ has no zero row, the matrix
$
	|A^{-1}|E_\varepsilon
$
	is positive, so its Perron vector is strictly positive. Repeating the above argument with full support (that is, with $k=n$) yields the desired estimate for $E_\varepsilon$, and the conclusion for $E$ then follows by a continuity argument as $\varepsilon\downarrow 0$.
	\end{rem}
	\begin{rem}
		In Proposition~\ref{prop:weighted-distance-improvement},	the refinement by $k_*$ is only nontrivial when $k_*<n$. This cannot occur if
		$B:=|A^{-1}|E$ is irreducible because then every Perron vector of $B$ is
		strictly positive, and hence $k_*=n$ by the Perron--Frobenius theorem.
		Therefore, the improvement from $n$ to $k_*$ is mainly relevant for structured
		or sparse weight matrices $E$, whose zero pattern may force $B$ to be reducible
		and allow Perron eigenvectors supported on proper subsets.
	\end{rem}
	\section{Improvements of selected results of B\"unger and Seeger}\label{sec:BS24}
	B\"unger and Seeger \cite{BS24} studied the sign-real spectral radius
	$
	\rho_{\R}(A)
	$
	and the associated set of sign-real expansive matrices
	\[
	\Omega_n:=\{A\in\R^{n\times n}:\rho_{\R}(A)\ge 1\}.
	\]
	A central theme of \cite[Section~5]{BS24} is to compare $\rho_{\R}$ with other functions that are easier to evaluate and enjoy similar invariance properties. In particular, they consider
	$
	\rho(|A|)
	$
	and, more generally, for $p\in(0,\infty]$,
	\[
	\rho_p^{\mathrm{abs}}(A):=
	\begin{cases}
		\rho\!\bigl(|A|^{[p]}\bigr)^{1/p}, & 0<p<\infty,\\[0.4em]
		\zeta(A), & p=\infty,
	\end{cases}
	\]
	where
	\[
	|A|^{[p]}:=\bigl(|a_{ij}|^p\bigr)_{i,j=1}^n
	\]
	denotes the entrywise $p$-th power of $|A|$, and $\zeta(A)$ is the maximum geometric mean of cyclic products introduced in \cite{Rum97b}.
	
	Among the main results of \cite[Section~5]{BS24} are the following. First, they \cite[Proposition~22]{BS24} proved that
	\begin{equation}\label{eq:BS24-prop22}
		\min\{1,n^{1/p-1}\}\,\rho_{\R}(A)
		\le
		\rho_p^{\mathrm{abs}}(A)
		\le
		(3+2\sqrt2)\,n^{1/p}\,\rho_{\R}(A)
		\qquad (A\in\R^{n\times n},\ p\in(0,\infty]).
	\end{equation}
	Second, they introduced the extremal constants
	\[
	\alpha_{n,p}:=\sup\{\rho_p^{\mathrm{abs}}(A):\ \rho_{\R}(A)=1\},
	\qquad
	\alpha_n:=\alpha_{n,1},
	\]
	and derived the general enclosure
	\begin{equation}\label{eq:BS24-alpha}
		(n-1)^{1/p}\le \alpha_{n,p}\le (3+2\sqrt2)\,n^{1/p},
	\end{equation}
	see \cite[eq.~(36)]{BS24}. Finally, they formulated Rump's conjecture~\eqref{eq:100}:
	\begin{equation}\label{eq:BS24-conj21}
		\rho(|A|)\le n\,\rho_{\R}(A)
		\qquad (A\in\R^{n\times n}),
	\end{equation}
	see \cite[Conjecture~21]{BS24}.
	
	Our Theorem~\ref{thm:global} proves \eqref{eq:BS24-conj21} and therefore sharpens several results in \cite{BS24}. We record these consequences below.
	
	\begin{thm}\label{thm:BS24-improvement}
		Let $A\in\R^{n\times n}$.
		
		\begin{enumerate}[label=\textup{(\alph*)}]
			\item One has
			\[
			\rho(|A|)\le n\,\rho_{\R}(A).
			\]
			
			\item For every $p\in(0,1]$,
			\[
			\rho_p^{\mathrm{abs}}(A)\le n^{1/p}\,\rho_{\R}(A).
			\]
			
			\item For every $p\in[1,\infty]$,
			\[
			\rho_p^{\mathrm{abs}}(A)\le
			\min\bigl\{n,\ (3+2\sqrt2)\,n^{1/p}\bigr\}\,\rho_{\R}(A).
			\]
		\end{enumerate}
	\end{thm}
	
	\begin{proof}
		Part \textup{(a)} is exactly the real case of Theorem~\ref{thm:global}.
		
		For \textup{(b)}, the case $p=1$ is just \textup{(a)}. Let $0<p<1$. By \cite[Proposition~26]{BS24}, for $0<p<1$ one has
		\[
		\rho_p^{\mathrm{abs}}(A)\le n^{1/p-1}\rho_1^{\mathrm{abs}}(A).
		\]
		Since $\rho_1^{\mathrm{abs}}(A)= \rho(|A|)$, part \textup{(a)} gives
		\[
		\rho_p^{\mathrm{abs}}(A)\le n^{1/p-1}\rho(|A|)\le n^{1/p}\rho_{\R}(A).
		\]
		
		For \textup{(c)}, if $p=1$, then the claim again reduces to \textup{(a)}. Let $1<p\le \infty$. By \cite[Proposition~26]{BS24},
		\[
		\rho_p^{\mathrm{abs}}(A)\le \rho_1^{\mathrm{abs}}(A)= \rho(|A|),
		\]
		and hence part \textup{(a)} yields
		\[
		\rho_p^{\mathrm{abs}}(A)\le n\,\rho_{\R}(A).
		\]
		On the other hand, the upper bound in \cite[Proposition~22]{BS24} gives
		\[
		\rho_p^{\mathrm{abs}}(A)\le (3+2\sqrt2)\,n^{1/p}\,\rho_{\R}(A).
		\]
		Taking the minimum of these two estimates proves \textup{(c)}.
	\end{proof}
	
	As an immediate corollary, one obtains improved bounds for the extremal constants $\alpha_{n,p}$.
	
	\begin{cor}\label{cor:BS24-alpha}
		Let
		$
		\alpha_{n,p}:=\sup\{\rho_p^{\mathrm{abs}}(A):\ \rho_{\R}(A)=1\}.
		$
		Then the following hold.
		
		\begin{enumerate}[label=\textup{(\alph*)}]
			\item For every $p\in(0,1]$,
			\[
			(n-1)^{1/p}\le \alpha_{n,p}\le n^{1/p}.
			\]
			
			\item In particular,
			\[
			\alpha_n=\alpha_{n,1}\le n.
			\]
			
			\item For every $p\in[1,\infty]$,
			\[
			\alpha_{n,p}\le \min\bigl\{n,\ (3+2\sqrt2)\,n^{1/p}\bigr\}.
			\]
		\end{enumerate}
	\end{cor}
	
	\begin{proof}
		The lower bound in \textup{(a)} is exactly the one proved in \cite[eq.~(36)]{BS24}. The upper bound in \textup{(a)} follows immediately from Theorem~\ref{thm:BS24-improvement}\textup{(b)} by restricting to matrices with $\rho_{\R}(A)=1$. Statement \textup{(b)} is the special case $p=1$. Statement \textup{(c)} follows similarly from Theorem~\ref{thm:BS24-improvement}\textup{(c)}.
	\end{proof}
	
	\begin{cor}\label{cor:BS24-inner}
		Let
		$
		\Gamma_{n,1}(\alpha):=\Bigl\{A\in\R^{n\times n}:\ \min_{1\le i\le n}\|Ae_i\|_1\ge \alpha\Bigr\}.
		$
		Then
		\[
		\Gamma_{n,1}(n)\subseteq \Omega_n.
		\]
	\end{cor}
	
	\begin{proof}
		By \cite[Lemma~24(b)]{BS24}, the inclusion
		\[
		\Gamma_{n,1}(\alpha)\subseteq \Omega_n
		\]
		is equivalent to the estimate
		\[
		\rho(|A|)\le \alpha\,\rho_{\R}(A)
		\qquad (A\in\R^{n\times n}).
		\]
		Applying Theorem~\ref{thm:BS24-improvement}\textup{(a)} with $\alpha=n$ gives the result.
	\end{proof}
	
	\begin{rem}
		Corollary~\ref{cor:BS24-inner} sharpens the technical inclusion used in the proof of \cite[Proposition~5(c)]{BS24}, where the threshold $(3+2\sqrt2)n$ appears. Thus, the path-connectedness argument for $\Omega_n$ in \cite{BS24} may be carried out with the smaller value $\alpha=n$.
		
		Corollary~\ref{cor:BS24-alpha}\textup{(a)} shows that, for $0<p\le 1$, the upper bound in \eqref{eq:BS24-alpha} can be improved from $(3+2\sqrt2)n^{1/p}$ to $n^{1/p}$. Since
		\[
		\frac{n^{1/p}}{(n-1)^{1/p}}\to 1
		\qquad (n\to\infty),
		\]
		this yields an asymptotically sharp enclosure for $\alpha_{n,p}$ throughout the range $0<p\le 1$.
		
		For $p>1$, Theorem~\ref{thm:BS24-improvement}\textup{(c)} does not uniformly improve the bound in \eqref{eq:BS24-prop22}, but it does give a better estimate whenever
		\[
		n^{1-1/p}\le 3+2\sqrt2.
		\]
		In particular, for $p=2$ one gets the sharper bound
		\[
		\rho_2^{\mathrm{abs}}(A)\le n\,\rho_{\R}(A)
		\]
		for all $n\le 33$.
	\end{rem}
	\section{Proof of Theorem~\ref{thm:euclid-sine}}\label{sec:euclid}
	
	Our proof of Theorem~\ref{thm:euclid-sine} relies on the following optimal plank theorem in Hilbert spaces due to Ortega-Moreno \cite[Theorem~1.3]{Ort21}.
	
	\begin{lem}[Ortega-Moreno]\label{lem:optimal}
		Let $H$ be a real Hilbert space, and let $v_1,\dots,v_n\in H$ be unit vectors. Then there exists a unit vector $v\in H$ such that
		\[
		|\langle v_k,v\rangle|\ge \sin\!\Bigl(\frac{\pi}{2n}\Bigr)
		\qquad (k=1,\dots,n).
		\]
	\end{lem}
	
	\begin{proof}[Proof of Theorem~\ref{thm:euclid-sine}]
		Write $A=(a_{ij})$. Let $H$ be the real Hilbert space underlying $\K^n$, endowed with the inner product
		\[
		\langle x,y\rangle_H:=\operatorname{Re}\!\left(\sum_{j=1}^n x_j\overline{y_j}\right).
		\]
		The induced norm is the usual Euclidean norm on $\K^n$.
		
		For each $i=1,\dots,n$, define
		\[
		u_i:=\frac{1}{\|r_i\|_2}\,(\overline{a_{i1}},\dots,\overline{a_{in}})^T\in H.
		\]
		Then $\|u_i\|_2=1$. By Lemma~\ref{lem:optimal}, there exists a unit vector $x\in H=\K^n$ such that
		\[
		|\langle u_i,x\rangle_H|\ge \sin\!\Bigl(\frac{\pi}{2n}\Bigr)
		\qquad (i=1,\dots,n).
		\]
		Since
		\[
		\langle u_i,x\rangle_H
		=
		\operatorname{Re}\!\left(\frac{\overline{(Ax)_i}}{\|r_i\|_2}\right)
		=
		\frac{\operatorname{Re}(Ax)_i}{\|r_i\|_2},
		\]
		we obtain
		\[
		|(Ax)_i|
		\ge
		|\operatorname{Re}(Ax)_i|
		=
		\|r_i\|_2\,|\langle u_i,x\rangle_H|
		\ge
		\sqrt{n-1}\sin\!\Bigl(\frac{\pi}{2n}\Bigr)
		\qquad (i=1,\dots,n).
		\]
		Hence
		\[
		\abs{Ax}\ge \sqrt{n-1}\sin\!\Bigl(\frac{\pi}{2n}\Bigr)e.
		\]
		
		Assume now that $\K=\R$. To see that the constant is optimal, consider the matrix $A_n\in\R^{n\times n}$ whose $k$th row is
		\[
		r_k:=\sqrt{n-1}\bigl(\cos((k-1)\pi/n),\,\sin((k-1)\pi/n),\,0,\dots,0\bigr),
		\qquad k=1,\dots,n.
		\]
		Then
		\[
		\norm{r_k}_2=\sqrt{n-1}
		\qquad (k=1,\dots,n),
		\]
		so $A_n$ satisfies the hypothesis of the theorem.
		
		Let $z=(z_1,\dots,z_n)\in\R^n$ be a unit vector, and set
		\[
		\rho:=\sqrt{z_1^2+z_2^2}\le 1.
		\]
		If $\rho=0$, then $A_nz=0$, and there is nothing to prove. Otherwise, write
		\[
		(z_1,z_2)=\rho(\cos\theta,\sin\theta)
		\]
		for some $\theta\in\R$. Then
		\[
		(A_nz)_k
		=
		\sqrt{n-1}\,\rho\cos\!\Bigl(\theta-\frac{(k-1)\pi}{n}\Bigr),
		\qquad k=1,\dots,n.
		\]
		Now set
		\[
		\beta_k:=\frac{\pi}{2}+\frac{(k-1)\pi}{n}
		\qquad (k=1,\dots,n),
		\]
		viewed modulo $\pi$. Since the points $\beta_1,\dots,\beta_n$ divide the circle of length $\pi$ into $n$ equal arcs, there exists $k$ such that
		\[
		\theta-\beta_k=\delta+m\pi
		\]
		for some integer $m$ and some real number $\delta$ satisfying
		\[
		|\delta|\le \frac{\pi}{2n}.
		\]
		For this $k$ we obtain
		\[
		|(A_nz)_k|
		=
		\sqrt{n-1}\,\rho\,
		\Bigl|\sin\!\bigl(\theta-\beta_k\bigr)\Bigr|
		=
		\sqrt{n-1}\,\rho\,|\sin\delta|
		\le
		\sqrt{n-1}\sin\!\Bigl(\frac{\pi}{2n}\Bigr).
		\]
		Hence
		\[
		\min_{1\le k\le n}|(A_nz)_k|
		\le
		\sqrt{n-1}\sin\!\Bigl(\frac{\pi}{2n}\Bigr)
		\]
		for every unit vector $z\in\R^n$. On the other hand, the first part of the theorem guarantees the reverse inequality for some unit vector $x\in\R^n$. Therefore
		\[
		\max_{\norm{z}_2=1}\min_{1\le k\le n}|(A_nz)_k|
		=
		\sqrt{n-1}\sin\!\Bigl(\frac{\pi}{2n}\Bigr),
		\]
		which proves optimality in the real case.
		
		The final assertion follows from $\abs{x}\le e$.
	\end{proof}
	
	Theorem~\ref{thm:euclid-sine} provides quantitative lower bounds for $\rho_{\R}(A)$ under the Euclidean row condition
	\[
	\norm{r_i}_2\ge \sqrt{n-1}
	\qquad (i=1,\dots,n).
	\]
	It is natural to ask whether these bounds can be upgraded to a genuine cube escape, or even to a normalized escape of the form $\abs{Ax}\ge e$. The next two propositions show that, in general, neither strengthening is available.
	
\begin{prop}[Failure of the $\ell_\infty$-strengthening]\label{prop:linf-counterexample}
	Consider the matrix
	\[
	A=\frac1{\sqrt2}
	\begin{pmatrix}
		1&1\\
		1&-1
	\end{pmatrix}.
	\]
	Then every row of $A$ has Euclidean norm $1=\sqrt{2-1}$, but there is no vector
	$x\in\R^2$ with $\norm{x}_\infty\le 1$ such that
	\[
	\abs{Ax}\ge e.
	\]
	Consequently, the implication
	\[
	\norm{r_i}_2\ge \sqrt{n-1}
	\quad (i=1,\dots,n)
	\qquad\Longrightarrow\qquad
	\exists x\in[-1,1]^n\text{ such that }\abs{Ax}\ge e
	\]
	fails already in dimension $2$.
\end{prop}

\begin{proof}
	A direct computation shows that each row of $A$ has Euclidean norm $1$.
	
	Suppose, for contradiction, that there exists
	\[
	x=(x_1,x_2)^T\in\R^2
	\qquad\text{with}\qquad
	\norm{x}_\infty\le 1
	\]
	such that
	\[
	\abs{Ax}\ge e.
	\]
	Write
	\[
	u:=x_1+x_2,
	\qquad
	v:=x_1-x_2.
	\]
	Then
	\[
	Ax=\frac1{\sqrt2}(u,v)^T,
	\]
	so $\abs{Ax}\ge e$ implies
	\[
	|u|\ge \sqrt2,
	\qquad
	|v|\ge \sqrt2.
	\]
	
	If $u$ and $v$ have the same sign, then
	\[
	|u+v|=|u|+|v|,
	\]
	and therefore
	\[
	|x_1|
	=
	\frac{|u+v|}{2}
	=
	\frac{|u|+|v|}{2}
	\ge \sqrt2>1,
	\]
	contrary to $\norm{x}_\infty\le 1$.
	
	If $u$ and $v$ have opposite signs, then
	\[
	|u-v|=|u|+|v|,
	\]
	and hence
	\[
	|x_2|
	=
	\frac{|u-v|}{2}
	=
	\frac{|u|+|v|}{2}
	\ge \sqrt2>1,
	\]
	again a contradiction.
	
	In both cases we obtain a contradiction. Hence no vector
	$x\in\R^2$ with $\norm{x}_\infty\le 1$ can satisfy $\abs{Ax}\ge e$.
\end{proof}
	\begin{rem}
		Proposition~\ref{prop:linf-counterexample} shows that the stronger estimate $|Ax|\ge e$, which was the key ingredient in establishing the sharp upper bound for the entrywise distance to the nearest singular matrix in Section~\ref{sec:distance}, fails under the weaker assumption $\|r_i\|_2\ge \sqrt{n-1}$.
	\end{rem}
	\begin{prop}[Failure of the $\ell_p$-strengthening for $1\le p<\infty$]\label{prop:lp-normalized-counterexample}
		Consider the matrix
		\[
		A=I_2.
		\]
		Then every row of $A$ has Euclidean norm $1$, but for every $1\le p<\infty$ there is no vector $x\in\R^2$ with $\norm{x}_p\le 1$ such that
		\[
		\abs{Ax}\ge e.
		\]
		Consequently, for every $1\le p<\infty$, the implication
		\[
		\norm{r_i}_2\ge \sqrt{n-1}
		\quad (i=1,\dots,n)
		\qquad\Longrightarrow\qquad
		\exists x\in\R^n\text{ with }\norm{x}_p\le 1\text{ and }\abs{Ax}\ge e
		\]
		fails already in dimension $2$.
	\end{prop}
	
	\begin{proof}
		Let $A=I_2$. Then each row has Euclidean norm $1=\sqrt{2-1}$. If there were a vector $x=(x_1,x_2)^T\in\R^2$ with $\norm{x}_p\le 1$ and $\abs{Ax}\ge e$, then $\abs{x_1}\ge 1$ and $\abs{x_2}\ge 1$, and therefore
		\[
		\norm{x}_p^p=\abs{x_1}^p+\abs{x_2}^p\ge 2.
		\]
		Hence $\norm{x}_p\ge 2^{1/p}>1$, a contradiction.
	\end{proof}

\section{Auxiliary Gaussian estimates}\label{sec:gaussian}

Throughout this section, let all random variables be defined on an underlying probability
space $(\Omega,\mathcal F,\PP)$.
We use the standard real and complex Gaussian conventions $N(0,\Sigma)$ and
$CN(0,\Sigma)$. The purpose of this section is twofold. First, we collect a few
elementary invariance and distributional facts for Gaussian vectors. Second, we
prove two two-dimensional scalar estimates, one in the complex setting and one
in the real setting, which bound the probability of the event
$
|U|<|V|
$
for a centered Gaussian pair $(U,V)$ in terms of $\EE |U|^2$. These bounds will
be the key probabilistic input in the covering arguments of
Section~\ref{sec:general-proof}.
Since only a few routine background facts are needed, we record them here and
defer standard material to Appendix~\ref{app:gaussian}.

\begin{lem}\label{lem:gaussian-facts}
	The following facts will be used repeatedly.
	\begin{enumerate}[label=\textup{(\alph*)}]
		\item If $g=(g_1,\ldots,g_m)^T\sim N(0,I_m)$ and $Q\in\R^{m\times m}$ is orthogonal, then
		\[
		Q^Tg\sim N(0,I_m).
		\]
		
		\item If $\zeta=(\zeta_1,\ldots,\zeta_m)^T\sim CN(0,I_m)$ and $Q\in\C^{m\times m}$ is unitary, then
		\[
		Q^*\zeta\sim CN(0,I_m).
		\]
		
		\item If $\xi\sim CN(0,1)$, then
		\[
		|\xi|^2\sim \operatorname{Exp}(1).
		\]
		
		\item If $Z_1$ and $Z_2$ are independent $N(0,1)$ random variables, then
		$
		Z_1/Z_2
		$
		has the standard Cauchy distribution. In particular, for every $s\ge 0$,
		\[
		\PP\!\left(\left|\frac{Z_1}{Z_2}\right|<s\right)
		=
		\frac{2}{\pi}\arctan s.
		\]
	\end{enumerate}
\end{lem}

 We say that a property holds \emph{almost surely} (abbreviated a.s.) if it
holds on an event of probability one. Equivalently, a statement depending on
$\omega\in\Omega$ holds almost surely if there exists an event
$A\in\mathcal F$ with $\PP(A)=1$ such that the statement holds for every
$\omega\in A$. We now state two scalar estimates.
	
	\begin{lem}[Complex scalar estimate]\label{lem:complex-scalar}
		Let
		\[
		\binom{U}{V}\sim CN(0,\Sigma),
		\qquad
		\Sigma=
		\begin{pmatrix}
			\alpha & c\\
			\overline{c} & 1
		\end{pmatrix},
		\qquad
		\alpha\ge 1,
		\qquad
		|c|^2\le \alpha.
		\]
		Then
		\begin{equation}\label{eq:complex-scalar}
			\PP\bigl(|U|<|V|\bigr)\le \frac{1}{1+\alpha}.
		\end{equation}
	\end{lem}
	
	\begin{proof}
		We first consider the degenerate case
	$
		|c|^2=\alpha.
$
		Since
		\[
		\EE|U|^2=\alpha,\qquad \EE(U\overline V)=c,\qquad \EE|V|^2=1,
		\]
		we compute
		\[
		\begin{aligned}
			\EE|U-cV|^2
			&=
			\EE|U|^2-\overline c\,\EE(U\overline V)-c\,\EE(V\overline U)+|c|^2\EE|V|^2\\
			&=
			\alpha-\overline c\,c-c\,\overline c+|c|^2\\
			&=
			\alpha-|c|^2\\
			&=0.
		\end{aligned}
		\]
		Now $|U-cV|^2$ is a nonnegative random variable. A nonnegative random
		variable with expectation $0$ must vanish almost surely. Hence
		\[
		U=cV
		\qquad\text{almost surely}.
		\]
		Since $\alpha=|c|^2\ge 1$, we have $|c|\ge 1$, and therefore
		\[
		|U|=|c|\,|V|\ge |V|
		\qquad\text{almost surely}.
		\]
		It follows that
		\[
		\PP(|U|<|V|)=0,
		\]
		which is stronger than \eqref{eq:complex-scalar}.
		
		Assume now that
	$
		|c|^2<\alpha.
$
		Then $\Sigma$ is Hermitian positive definite. Let
	$
		J=\diag(1,-1).
	$
		Choose $\zeta\sim CN(0,I_2)$ and write
		\[
		\binom{U}{V}=\Sigma^{1/2}\zeta.
		\]
		Since
		\[
		|U|^2-|V|^2
		=
		\begin{pmatrix}\overline U&\overline V\end{pmatrix}
		\begin{pmatrix}1&0\\0&-1\end{pmatrix}
		\binom{U}{V},
		\]
		we obtain
		\[
		|U|^2-|V|^2
		=
		\zeta^*\bigl(\Sigma^{1/2}J\Sigma^{1/2}\bigr)\zeta.
		\]
		
		Set
$
		H:=\Sigma^{1/2}J\Sigma^{1/2}.
$
		Since $\Sigma$ is invertible,
		\[
		\Sigma^{-1/2}H\Sigma^{1/2}=J\Sigma,
		\]
	so $H$ is similar to $J\Sigma$ and hence has the same eigenvalues as
		$J\Sigma$.
		The characteristic polynomial of $J\Sigma$ is
		\[
		\det(\lambda I-J\Sigma)
		=
		\lambda^2-(\alpha-1)\lambda+(|c|^2-\alpha).
		\]
		Therefore the eigenvalues of $H$ are
		\[
		\lambda_{\pm}
		=
		\frac{\alpha-1\pm \Delta}{2},
		\qquad
		\Delta:=\sqrt{(\alpha+1)^2-4|c|^2}.
		\]
		Moreover,
		\[
		\det(J\Sigma)=|c|^2-\alpha<0,
		\]
		so the two eigenvalues have opposite signs. Since $H$ is Hermitian, its
		eigenvalues are real; therefore
		\[
		\lambda_+>0>\lambda_-.
		\]
		
	Since $H$ is Hermitian, there exists a unitary matrix $Q$ such that
		\[
		Q^*HQ=\diag(\lambda_+,\lambda_-).
		\]
		Write
		\[
		Q^*\zeta=(\xi_1,\xi_2)^T.
		\]
		As noted in Lemma~\ref{lem:gaussian-facts}, $Q^*\zeta\sim CN(0,I_2)$. Hence $\xi_1$ and $\xi_2$ are
		independent standard complex Gaussian random variables. By using Lemma~\ref{lem:gaussian-facts} again,
		$|\xi_1|^2$ and $|\xi_2|^2$ are independent $\operatorname{Exp}(1)$ random
		variables. Set
		\[
		X:=|\xi_1|^2,\qquad Y:=|\xi_2|^2.
		\]
		Then $X$ and $Y$ are independent and each has density
		\[
		e^{-t}\mathbf 1_{[0,\infty)}(t).
		\]
		
		Now
		\[
		\begin{aligned}
			|U|^2-|V|^2
			&=
			\zeta^*H\zeta\\
			&=
			\zeta^*Q\diag(\lambda_+,\lambda_-)Q^*\zeta\\
			&=
			(Q^*\zeta)^*\diag(\lambda_+,\lambda_-) (Q^*\zeta)\\
			&=
			\lambda_+|\xi_1|^2+\lambda_-|\xi_2|^2\\
			&=
			\lambda_+X+\lambda_-Y.
		\end{aligned}
		\]
		Hence
		\[
		\PP(|U|<|V|)
		=
		\PP(|U|^2-|V|^2<0)
		=
		\PP(\lambda_+X<-\lambda_-Y).
		\]
		
		We next compute $	\PP(aX<bY)$ for arbitrary $a,b>0$,
		Since $X$ and $Y$ are independent, conditioning on the value of $Y$ gives
		\[
		\PP(aX<bY)
		=
		\int_0^\infty \PP\!\left(X<\frac{b}{a}y\right)e^{-y}\,dy.
		\]
		Because $X\sim \operatorname{Exp}(1)$,
		\[
		\PP(X<t)=1-e^{-t}
		\qquad (t\ge 0).
		\]
		Therefore
		\[
		\begin{aligned}
			\PP(aX<bY)
			&=
			\int_0^\infty \left(1-e^{-(b/a)y}\right)e^{-y}\,dy\\
			&=
			\int_0^\infty e^{-y}\,dy-\int_0^\infty e^{-(1+b/a)y}\,dy\\
			&=
			1-\frac{1}{1+b/a}\\
			&=
			\frac{b}{a+b}.
		\end{aligned}
		\]
		Applying this with
		\[
		a=\lambda_+,\qquad b=-\lambda_->0,
		\]
		we obtain
		\[
		\PP(|U|<|V|)
		=
		\frac{-\lambda_-}{\lambda_+-\lambda_-}
		=
		\frac{\Delta-(\alpha-1)}{2\Delta}.
		\]
		
		Subtracting $1/(1+\alpha)$ from both sides gives
\begin{equation}\label{eq:right-hand}
		\PP(|U|<|V|)-\frac{1}{1+\alpha}
	=
	-\frac{(\alpha-1)\bigl((\alpha+1)-\Delta\bigr)}{2\Delta(1+\alpha)}.
\end{equation}
		Since
		\[
		\alpha\ge 1
		\qquad\text{and}\qquad
		\Delta\le \alpha+1,
		\]
		the right-hand side of \eqref{eq:right-hand} is nonpositive. Therefore
		\[
		\PP(|U|<|V|)\le \frac{1}{1+\alpha}.
		\]
		This proves \eqref{eq:complex-scalar}.
	\end{proof}
	
	\begin{lem}[Real scalar estimate]\label{lem:real-scalar}
		Let
		\[
		\binom{U}{V}\sim N(0,\Sigma),
		\qquad
		\Sigma=
		\begin{pmatrix}
			\alpha & c\\
			c & 1
		\end{pmatrix},
		\qquad
		\alpha\ge 1,
		\qquad
		c^2\le \alpha.
		\]
		Then
		\begin{equation}\label{eq:real-scalar}
			\PP\bigl(|U|<|V|\bigr)
			\le
			\frac{2}{\pi}\arctan\!\Bigl(\frac{1}{\sqrt{\alpha}}\Bigr).
		\end{equation}
	\end{lem}
	
	\begin{proof}
		We first treat the degenerate case
$
		c^2=\alpha.
$
		Exactly as in the proof of Lemma~\ref{lem:complex-scalar}, one finds
		\[
		\EE|U-cV|^2=0,
		\]
		hence
		\[
		U=cV
		\qquad\text{almost surely}.
		\]
		Since $|c|=\sqrt{\alpha}\ge 1$, it follows that
		\[
		|U|=|c|\,|V|\ge |V|
		\qquad\text{almost surely},
		\]
		and therefore
		\[
		\PP(|U|<|V|)=0.
		\]
		
		Assume now that
	$
		c^2<\alpha.
$
		Then $\Sigma$ is symmetric positive definite. Let
		\[
		J=\diag(1,-1).
		\]
		Choose $\zeta\sim N(0,I_2)$ and write
		\[
		\binom{U}{V}=\Sigma^{1/2}\zeta.
		\]
		As in the complex case,
		\[
		|U|^2-|V|^2=\zeta^T H\zeta,
		\qquad
		H:=\Sigma^{1/2}J\Sigma^{1/2}.
		\]
		The matrix $H$ is real symmetric. Since
		\[
		\Sigma^{-1/2}H\Sigma^{1/2}=J\Sigma,
		\]
		$H$ is similar to $J\Sigma$, so the two matrices have the same
		eigenvalues. The characteristic polynomial of $J\Sigma$ is
		\[
		\lambda^2-(\alpha-1)\lambda+(c^2-\alpha),
		\]
		hence the eigenvalues of $H$ are
		\[
		\lambda_{\pm}
		=
		\frac{\alpha-1\pm \Delta}{2},
		\qquad
		\Delta:=\sqrt{(\alpha+1)^2-4c^2}.
		\]
		Since
		\[
		\det(J\Sigma)=c^2-\alpha<0,
		\]
		these eigenvalues are real and of opposite signs:
		\[
		\lambda_+>0>\lambda_-.
		\]
		
		Choose an orthogonal matrix $Q$ such that
		\[
		Q^T H Q=\diag(\lambda_+,\lambda_-).
		\]
		Write
		\[
		Q^T\zeta=(Z_1,Z_2)^T.
		\]
		Because the density of $N(0,I_2)$ is
		\[
		(2\pi)^{-1}e^{-\|x\|_2^2/2},
		\]
		which depends only on $\|x\|_2$, it is invariant under orthogonal
		transformations. Thus
		\[
		Q^T\zeta\sim N(0,I_2),
		\]
		so $Z_1$ and $Z_2$ are independent $N(0,1)$ random variables.
		
		Exactly as before,
		\[
		\begin{aligned}
			|U|^2-|V|^2
			&=
			\zeta^T H\zeta\\
			&=
			\zeta^TQ\diag(\lambda_+,\lambda_-)Q^T\zeta\\
			&=
			(Q^T\zeta)^T\diag(\lambda_+,\lambda_-)(Q^T\zeta)\\
			&=
			\lambda_+Z_1^2+\lambda_-Z_2^2.
		\end{aligned}
		\]
		Hence
		\[
		\PP(|U|<|V|)
		=
		\PP(\lambda_+Z_1^2<-\lambda_-Z_2^2)
		=
		\PP\!\left(\left|\frac{Z_1}{Z_2}\right|<\sqrt{\frac{-\lambda_-}{\lambda_+}}\right).
		\]
	By Lemma~\ref{lem:gaussian-facts}, $Z_1/Z_2$ has the standard Cauchy distribution,
		\[
		\PP(|U|<|V|)
		=
		\frac{2}{\pi}\arctan\!\sqrt{\frac{-\lambda_-}{\lambda_+}}.
		\]
		
		Now
		\[
		\frac{-\lambda_-}{\lambda_+}
		=
		\frac{\Delta-(\alpha-1)}{\Delta+(\alpha-1)}.
		\]
		We claim that
		\[
		\frac{-\lambda_-}{\lambda_+}\le \frac{1}{\alpha}.
		\]
		Indeed,
		\[
		\frac{\Delta-(\alpha-1)}{\Delta+(\alpha-1)}\le \frac{1}{\alpha}
		\iff
		\alpha\bigl(\Delta-(\alpha-1)\bigr)\le \Delta+(\alpha-1),
		\]
		which is equivalent to
		\[
		(\alpha-1)\bigl(\Delta-(\alpha+1)\bigr)\le 0.
		\]
		This is true because
		\[
		\alpha\ge 1
		\qquad\text{and}\qquad
		\Delta\le \alpha+1.
		\]
		Therefore
		\[
		\sqrt{\frac{-\lambda_-}{\lambda_+}}\le \frac{1}{\sqrt{\alpha}}.
		\]
		Since $\arctan$ is increasing on $[0,\infty)$, we conclude that
		\[
		\PP(|U|<|V|)
		\le
		\frac{2}{\pi}\arctan\!\Bigl(\frac{1}{\sqrt{\alpha}}\Bigr).
		\]
		This proves \eqref{eq:real-scalar}.
	\end{proof}

\section{Proofs of Theorems~\ref{thm:complex-general} and \ref{thm:real-general}}\label{sec:general-proof}
In this section, we prove Theorems~\ref{thm:complex-general} and \ref{thm:real-general} by using Lemmas~\ref{lem:complex-scalar} and \ref{lem:real-scalar}. A further ingredient is the following elementary topological observation: if $n\ge 2$, then the punctured spaces $\C^n\setminus\{0\}$ and $\R^n\setminus\{0\}$ are path connected. Consequently, they are connected and therefore cannot be decomposed into a disjoint union of nonempty open subsets; see~\cite[Chapter~3, \S23]{Mun00}.
	\begin{proof}[Proof of Theorem~\ref{thm:complex-general}]
		Assume, for contradiction, that there is no nonzero vector $x\in \C^n$ such that $\abs{Ax}\ge \abs{x}$. Then every nonzero vector $z\in \C^n$ belongs to at least one of the sets
		\[
		E_i:=\bigl\{z\in \C^n : \abs{(Az)_i}<\abs{z_i}\bigr\}
		\qquad (1\le i\le n).
		\]
		In other words,
		\begin{equation}\label{eq:complex-cover}
			\C^n\setminus \{0\}=\bigcup_{i=1}^n E_i.
		\end{equation}
		Each $E_i$ is open, because the map
		$
		z\mapsto |(Az)_i|-|z_i|
		$
		is continuous and
		$
		E_i=\{z\in\C^n:\ |(Az)_i|-|z_i|<0\}
		$
		is the inverse image of the open interval $(-\infty,0)$.
		
		Let $g=(g_1,\dots,g_n)^T\sim CN(0,I_n)$. Since $\PP(g=0)=0$, \eqref{eq:complex-cover} implies
		\[
		1=\PP\!\left(\bigcup_{i=1}^n E_i\right).
		\]
		Fix $i\in\{1,\dots,n\}$, and define
		\[
		U_i:=(Ag)_i,
		\qquad
		V_i:=g_i,
		\qquad
		\alpha_i:=\norm{r_i}_2^2,
		\qquad
		c_i:=a_{ii}.
		\]
		Then $(U_i,V_i)$ is a centered proper complex Gaussian vector in $\C^2$ with covariance matrix
		\[
		\Sigma_i=
		\begin{pmatrix}
			\alpha_i & c_i\\
			\overline{c_i} & 1
		\end{pmatrix}.
		\]
Furthermore,
$
\alpha_i=\norm{r_i}_2^2\ge 1
$
by hypothesis, and
\[
|c_i|^2=|a_{ii}|^2\le \sum_{j=1}^n |a_{ij}|^2=\norm{r_i}_2^2=\alpha_i.
\]
Hence Lemma~\ref{lem:complex-scalar} applies and gives
		\[
		\PP(E_i)
		=
		\PP\bigl(\abs{U_i}<\abs{V_i}\bigr)
		\le
		\frac{1}{1+\alpha_i}
		=
		\frac{1}{1+\norm{r_i}_2^2}.
		\]
		Summing over $i$, we obtain
		\[
		1
		=
		\PP\!\left(\bigcup_{i=1}^n E_i\right)
		\le
		\sum_{i=1}^n \PP(E_i)
		\le
		\sum_{i=1}^n \frac{1}{1+\norm{r_i}_2^2}
		\le 1.
		\]
		Therefore all inequalities in the preceding display are equalities.
		
		We now show that each $E_i$ is nonempty. Indeed, if $E_k=\varnothing$ for some $k$, then
		\[
		1
		=
		\PP\!\left(\bigcup_{i=1}^n E_i\right)
		\le
		\sum_{i\ne k} \PP(E_i)
		\le
		\sum_{i\ne k} \frac{1}{1+\norm{r_i}_2^2}
		<
		\sum_{i=1}^n \frac{1}{1+\norm{r_i}_2^2}
		\le 1,
		\]
		a contradiction. Hence every $E_i$ is nonempty.
		
		Next we prove that the family $(E_i)_{i=1}^n$ is pairwise disjoint. For \(i\in\{1,\dots,n\}\), let
		\[
		H_i:=\bigcup_{j<i} E_j,
		\]
		with the convention \(H_1=\varnothing\). Since the sets \(E_i\setminus H_i\) are pairwise disjoint and their union is \(\bigcup_{i=1}^n E_i\), we have
		\[
		\PP\!\left(\bigcup_{i=1}^n E_i\right)
		=
		\sum_{i=1}^n \PP(E_i\setminus H_i).
		\]
	Since \(E_i\setminus H_i = E_i \setminus (E_i\cap H_i)\) and \(H_i\) is measurable, it follows that
		\[
		\PP(E_i\setminus H_i)=\PP(E_i)-\PP(E_i\cap H_i)
		\qquad (2\le i\le n).
		\]
		Hence
		\[
		\sum_{i=2}^n \PP(E_i\cap H_i)
		=
		\sum_{i=1}^n \PP(E_i)-\PP\!\left(\bigcup_{i=1}^n E_i\right)
		=
		0.
		\]
		Therefore
		\[
		\PP(E_i\cap H_i)=0
		\qquad (2\le i\le n).
		\]
		Now the law of $g$ has a strictly positive smooth density on $\C^n$, so every nonempty open subset of $\C^n$ has strictly positive probability. Since each $E_i\cap H_i$ is open, we conclude that
		\[
		E_i\cap H_i=\varnothing
		\qquad (2\le i\le n).
		\]
		Thus $E_1,\dots,E_n$ are pairwise disjoint.
		
		We have shown that $\C^n\setminus\{0\}$ is the disjoint union of the nonempty open sets $E_1,\dots,E_n$. This contradicts the path connectedness of $\C^n\setminus\{0\}$. The contradiction proves the theorem.
	\end{proof}
	\begin{proof}[Proof of Theorem~\ref{thm:real-general}]
		Assume, for contradiction, that there is no nonzero vector $x\in \R^n$ such that $\abs{Ax}\ge \abs{x}$ coordinatewise. Define
		\[
		E_i:=\bigl\{z\in \R^n : \abs{(Az)_i}<\abs{z_i}\bigr\}
		\qquad (1\le i\le n).
		\]
		Then each $E_i$ is open and
		\begin{equation}\label{eq:real-cover}
			\R^n\setminus\{0\}=\bigcup_{i=1}^n E_i.
		\end{equation}
		
		Let $g=(g_1,\dots,g_n)^T\sim N(0,I_n)$. Since $\PP(g=0)=0$, \eqref{eq:real-cover} implies
		\[
		1=\PP\!\left(\bigcup_{i=1}^n E_i\right).
		\]
		Fix $i\in\{1,\dots,n\}$, and define
		\[
		U_i:=(Ag)_i,
		\qquad
		V_i:=g_i,
		\qquad
		\alpha_i:=\norm{r_i}_2^2,
		\qquad
		c_i:=a_{ii}.
		\]
		Then $(U_i,V_i)$ is a centered real Gaussian vector in $\R^2$ with covariance matrix
		\[
		\Sigma_i=
		\begin{pmatrix}
			\alpha_i & c_i\\
			c_i & 1
		\end{pmatrix}.
		\]
		By Lemma~\ref{lem:real-scalar},
		\[
		\PP(E_i)
		=
		\PP\bigl(\abs{U_i}<\abs{V_i}\bigr)
		\le
		\frac{2}{\pi}\arctan\!\Bigl(\frac{1}{\sqrt{\alpha_i}}\Bigr)
		=
		\frac{2}{\pi}\arctan\!\Bigl(\frac{1}{\norm{r_i}_2}\Bigr).
		\]
		Therefore
		\[
		1
		=
		\PP\!\left(\bigcup_{i=1}^n E_i\right)
		\le
		\sum_{i=1}^n \PP(E_i)
		\le
		\sum_{i=1}^n \frac{2}{\pi}\arctan\!\Bigl(\frac{1}{\norm{r_i}_2}\Bigr)
		\le 1.
		\]
		Hence all inequalities are equalities.
		
		Exactly as in the proof of Theorem~\ref{thm:complex-general}, equality implies first that each $E_i$ is nonempty, and second that the sets $E_1,\dots,E_n$ are pairwise disjoint. Consequently $\R^n\setminus\{0\}$ is the disjoint union of $n$ nonempty open sets, contradicting the path connectedness of $\R^n\setminus\{0\}$ for $n\ge 2$. This contradiction proves the theorem.
	\end{proof}
	
	\begin{rem}
		The proof of Theorem~\ref{thm:complex-general} is genuinely stronger than the real one, because the complex scalar estimate in  Lemma~\ref{lem:complex-scalar} is sharper than its real counterpart Lemma~\ref{lem:real-scalar}. This is exactly why the natural complex threshold $\sqrt{n-1}$ is not reached by the same argument in the real case.
	\end{rem}
	
	\section{Further remarks on Conjecture~\ref{conj:200}}\label{sec:further}
	
	In this section, we record an additional positive result for the Euclidean row condition in the planar case. Although Conjecture~\ref{conj:200} remains open in general, the two-dimensional case already follows from Theorem~\ref{thm:real-general}. Nevertheless, it admits a more direct argument based on a simple determinant criterion.
	
	We begin with an elementary lemma that converts a suitable signed spectral condition into the desired componentwise inequality. This is one of the many characterizations of the sign-real spectral radius; see \cite[Theorem~2.3]{Rum97b}. For the reader's convenience we include a short proof.
	\begin{lem}\label{lem:signature}
		Let $A\in\R^{2\times 2}$. Suppose there exists a signature matrix
		\[
		D=\diag(\varepsilon_1,\varepsilon_2),
		\qquad \varepsilon_1,\varepsilon_2\in\{\pm1\},
		\]
		such that
		\[
		\det(I-DA)\le 0.
		\]
		Then there exists $x\in\R^2\setminus\{0\}$ such that
		\[
		\abs{Ax}\ge \abs{x}.
		\]
	\end{lem}
	
	\begin{proof}
		Let
		\[
		p(\lambda):=\det(\lambda I-DA).
		\]
		Then $p$ is a real quadratic polynomial with positive leading coefficient, and
		\[
		p(1)=\det(I-DA)\le 0.
		\]
		It follows that $p(\lambda)=0$ for some real $\lambda\ge 1$. Let $x\ne 0$ be an eigenvector of $DA$ corresponding to $\lambda$. Then
		\[
		\abs{Ax}=\abs{DAx}=\lambda\abs{x}\ge \abs{x},
		\]
		which proves the claim.
	\end{proof}
	
	We now apply Lemma~\ref{lem:signature} to show that Conjecture~\ref{conj:200} is true in dimension $2$. Thus, while the higher-dimensional problem appears to be rather delicate, the planar case can  be settled by an explicit argument.
	
	\begin{thm}\label{thm:planar}
		Let $A\in\R^{2\times 2}$, and let $r_1,r_2$ denote the rows of $A$. If
		\[
		\norm{r_i}_2\ge 1
		\qquad (i=1,2),
		\]
		then there exists a vector $x\in\R^2\setminus\{0\}$ such that
		\[
		\abs{Ax}\ge \abs{x}.
		\]
	\end{thm}
	
	\begin{proof}
		Write
		\[
		A=\begin{pmatrix} a & b \\ c & d \end{pmatrix}.
		\]
		If $\abs{a}\ge 1$, then taking $x=(1,0)$ yields
		\[
		\abs{Ax}=(\abs{a},\abs{c})^T\ge (1,0)^T=\abs{x}.
		\]
		Similarly, if $\abs{d}\ge 1$, then $x=(0,1)$ works. We may therefore assume that
		\[
		\abs{a}<1,
		\qquad
		\abs{d}<1.
		\]
		
		Since $\norm{r_1}_2\ge 1$ and $\norm{r_2}_2\ge 1$, we obtain
		\[
		\abs{b}\ge \sqrt{1-a^2}>0,
		\qquad
		\abs{c}\ge \sqrt{1-d^2}>0.
		\]
		In particular, $bc\ne 0$.
		
		Set
		\[
		s:=\sgn(bc)\in\{\pm1\}.
		\]
		Among the two pairs $(\varepsilon_1,\varepsilon_2)\in\{\pm1\}^2$ satisfying $\varepsilon_1\varepsilon_2=s$, choose one that minimizes
		\[
		(1-\varepsilon_1 a)(1-\varepsilon_2 d).
		\]
		If $s=1$, the two possible values are
		\[
		(1-a)(1-d)
		\qquad\text{and}\qquad
		(1+a)(1+d),
		\]
		whereas if $s=-1$, they are
		\[
		(1-a)(1+d)
		\qquad\text{and}\qquad
		(1+a)(1-d).
		\]
		In either case, the product of the two admissible values is
		\[
		(1-a^2)(1-d^2).
		\]
		Hence the smaller one does not exceed the geometric mean, and therefore
		\begin{equation}\label{eq:min-planar}
			(1-\varepsilon_1 a)(1-\varepsilon_2 d)
			\le \sqrt{(1-a^2)(1-d^2)}.
		\end{equation}
		
		On the other hand,
		\[
		\abs{bc}\ge \sqrt{1-a^2}\,\sqrt{1-d^2}
		= \sqrt{(1-a^2)(1-d^2)}.
		\]
		Combining this with \eqref{eq:min-planar}, we obtain
		\[
		(1-\varepsilon_1 a)(1-\varepsilon_2 d)\le \abs{bc}.
		\]
		
		Now let
		\[
		D=\diag(\varepsilon_1,\varepsilon_2).
		\]
		Since $\varepsilon_1\varepsilon_2=s=\sgn(bc)$, we have
		\[
		\varepsilon_1\varepsilon_2bc=\abs{bc},
		\]
		and hence
		\[
		\det(I-DA)
		=(1-\varepsilon_1 a)(1-\varepsilon_2 d)-\varepsilon_1\varepsilon_2bc
		=(1-\varepsilon_1 a)(1-\varepsilon_2 d)-\abs{bc}\le 0.
		\]
		Lemma~\ref{lem:signature} gives the desired conclusion.
	\end{proof}
As noted above, the complex analogue of Conjecture~\ref{conj:200} is contained in Theorem~\ref{thm:complex-general}. In the planar complex case, however, the conclusion can be strengthened further.
	\begin{thm}\label{thm:planar-complex}
		Let $A\in\C^{2\times 2}$, and let $r_1,r_2$ denote the rows of $A$. If
		\[
		\norm{r_i}_2\ge 1
		\qquad (i=1,2),
		\]
		then there exists a vector $x\in\C^2\setminus\{0\}$ with $|x|=e$ such that
		\[
		\abs{Ax}\ge e.
		\]
	\end{thm}
	
	\begin{proof}
		Write
		\[
		A=\begin{pmatrix} a & b \\ c & d \end{pmatrix},
		\qquad
		\mathbb T:=\{z\in\C:\ |z|=1\}.
		\]
		For the two rows of $A$, define
		\[
		E_1:=\{z\in\mathbb T:\ |a+bz|\ge 1\},
		\qquad
		E_2:=\{z\in\mathbb T:\ |c+dz|\ge 1\}.
		\]
		We claim that each $E_i$ is either all of $\mathbb T$ or a closed arc of angular length at least $\pi$.
		
		Indeed, consider $E_1$. If $ab=0$, then since
		\[
		|a|^2+|b|^2=\|r_1\|_2^2\ge 1,
		\]
		one has either $|a|\ge 1$ or $|b|\ge 1$, and therefore $E_1=\mathbb T$.
		Now assume that $ab\ne 0$. Write
		\[
		a=\alpha e^{i\mu},\qquad b=\beta e^{i\nu},
		\qquad \alpha,\beta>0.
		\]
		Then, as $z$ runs through $\mathbb T$, so does $e^{i(\nu-\mu)}z$, and
		\[
		|a+bz|
		=
		\bigl|\alpha+\beta e^{i(\nu-\mu)}z\bigr|.
		\]
		Hence $E_1$ is a rotation of the set of all $e^{it}\in\mathbb T$ such that
		\[
		|\alpha+\beta e^{it}|\ge 1.
		\]
		Squaring both sides, this is equivalent to
		\[
		\alpha^2+\beta^2+2\alpha\beta\cos t\ge 1,
		\]
		that is,
		\[
		\cos t\ge \gamma:=\frac{1-\alpha^2-\beta^2}{2\alpha\beta}.
		\]
		Since $\alpha^2+\beta^2=\|r_1\|_2^2\ge 1$, one has $\gamma\le 0$.
		
		If $\gamma\le -1$, then $\cos t\ge \gamma$ for every $t$, and therefore $E_1=\mathbb T$.
		If $-1<\gamma\le 0$, then the admissible set of angles is the closed interval
		\[
		[-\arccos(\gamma),\,\arccos(\gamma)],
		\]
		whose angular length is
		\[
		2\arccos(\gamma)\ge \pi.
		\]
		Thus in all cases $E_1$ is either all of $\mathbb T$ or a closed arc of angular length at least $\pi$.
		The same argument applies to $E_2$.
		
		Now any two closed arcs of the unit circle having angular length at least $\pi$ must intersect. Hence
		\[
		E_1\cap E_2\ne \varnothing.
		\]
		Choose $z\in E_1\cap E_2$ and set
		\[
		x:=(1,z)^T\in\C^2.
		\]
		Then $x\ne 0$, $|z|=1$, and thus
		\[
		|x|=e.
		\]
		Moreover,
		\[
		|Ax|
		=
		\bigl(|a+bz|,\ |c+dz|\bigr)^T
		\ge e.
		\]
		This proves the claim.
	\end{proof}
	
	\section*{Acknowledgments.}
	The author thanks Professor Siegfried M.~Rump for very helpful comments on several  earlier versions of the manuscript. He also thanks Professor Minghua Lin for introducing this problem to him in 2021. This work is supported by the China Scholarship Council, the Young Elite Scientists Sponsorship Program for PhD Students (China Association for Science and Technology), and the Fundamental Research Funds for the Central Universities at Xi'an Jiaotong University (Grant No.~xzy022024045).

	\appendix
	\section{Proof of Theorem~\ref{thm:global}}\label{sec:proof-compare}
	\begin{proof}[Proof of Theorem~\ref{thm:global}]
		Let $\K\in\{\R,\C\}$ and  $A\in\K^{n\times n}$. Set
		\[
		B:=\abs{A},
		\qquad
		\alpha:=\rhoo(B).
		\]
		If $\alpha=0$, then the conclusion is trivial. Assume henceforth that $\alpha>0$.
		
		By the Frobenius normal form (for example, see \cite{BP94} or \cite{Sen06}), there exists a nonempty index set
		$
		\mu\subseteq\{1,\dots,n\}
		$
		such that $B[\mu]$ is irreducible and
		\[
		\rhoo(B[\mu])=\rhoo(B)=\alpha.
		\]
		Write $k:=|\mu|$. Since $B[\mu]$ is nonnegative and irreducible, the Perron--Frobenius theorem yields a vector
		\[
		z=(z_1,\dots,z_k)\in\R_{>0}^k
		\]
		such that
		\[
		B[\mu]z=\alpha z.
		\]
		Define
		\[
		D:=\diag(z_1,\dots,z_k)
		\qquad\text{and}\qquad
		C:=D^{-1}A[\mu]D.
		\]
		Because $D$ has positive diagonal entries,
		\[
		\abs{C}e
		=
		D^{-1}\abs{A[\mu]}De
		=
		D^{-1}B[\mu]z
		=
		\alpha e.
		\]
		
		Apply Theorem~\ref{thm:mainp} with $p=\infty$ and $m=n=k$ to the matrix $C\in\K^{k\times k}$ and with
		$
		t=\alpha/k.
		$
		Since every row of $C$ has $\ell_1$-norm equal to $\alpha$, the hypotheses are satisfied. Hence there exists
		$
		x\in\K^k\setminus\{0\}
		$
		with $\norm{x}_\infty\le 1$ such that
		\[
		\abs{Cx}\ge \frac{\alpha}{k}\,e.
		\]
		As $\abs{x}\le e$, it follows that
		\[
		\abs{Cx}\ge \frac{\alpha}{k}\,\abs{x}.
		\]
		By the Collatz--Wielandt type characterization \eqref{eq:intro-cw},
		\[
		\rho_{\K}(C)\ge \frac{\alpha}{k}.
		\]
		
		Finally, by properties \textup{(vi)} and \textup{(iv)},
		\[
		\rho_{\K}(A)\ge \rho_{\K}(A[\mu])=\rho_{\K}(C)\ge \frac{\alpha}{k}\ge \frac{\alpha}{n}.
		\]
		Therefore,
		\[
		\rhoo(\abs{A})=\alpha\le n\,\rho_{\K}(A).
		\]
	\end{proof}
	
\section{Gaussian background}\label{app:gaussian}

For the convenience of the reader, we collect here the conventions and
elementary facts on real and complex Gaussian variables used in
Section~\ref{sec:gaussian}. Standard references for the material below include
Durrett~\cite{Dur19} for measure-theoretic probability, Gaussian
vectors, and the Cauchy ratio identity,
and Goodman~\cite{Goo63} (see also
\cite[Chapter~3, \S3.1--\S3.4]{MPH22}) for proper complex
Gaussian vectors.

Let all random variables be defined on an underlying probability
space $(\Omega,\mathcal F,\PP)$, where $\Omega$ is the sample space,
$\mathcal F$ is a $\sigma$-algebra of events, and $\PP$ is a probability
measure on $(\Omega,\mathcal F)$. If $A\in\mathcal F$ is an event, then
$\PP(A)$ denotes its probability. If $X$ is an integrable real-valued random
variable, that is, a measurable function $X:\Omega\to\R$ satisfying
$
\int_\Omega |X(\omega)|\,d\PP(\omega)<\infty,
$
then its expectation is defined by
\[
\EE X=\int_\Omega X(\omega)\,d\PP(\omega),
\]
where $\omega\in\Omega$ denotes a sample point.

\subsection{Densities on $\R^m$ and $\C^m$}

Let $X$ be a random vector taking values in $\R^m$. We say that $X$ has
\emph{density} $f:\R^m\to[0,\infty)$ if for every Borel set
$E\subseteq \R^m$,
\[
\PP(E):=\PP(X\in E)=\int_E f(x)\,dx,
\]
where $dx$ denotes Lebesgue measure on $\R^m$. In particular,
$
\int_{\R^m} f(x)\,dx=1.
$

Likewise, if $Z$ is a random vector taking values in $\C^m$, we identify
$\C^m$ with $\R^{2m}$ in the standard way and say that $Z$ has density
$p:\C^m\to[0,\infty)$ if for every Borel set $E\subseteq \C^m$,
\[
\PP(E):=\PP(Z\in E)=\int_E p(z)\,dz,
\]
where $dz$ denotes Lebesgue measure on $\C^m\cong\R^{2m}$. Again,
$
\int_{\C^m} p(z)\,dz=1.
$

In the scalar complex case, if $z=x+iy\in\C$, then $dz=dx\,dy$. Passing to
polar coordinates $z=re^{i\theta}$ gives
\[
dz=r\,dr\,d\theta.
\]

\subsection{Real and complex Gaussian vectors}

A random vector
\[
g=(g_1,\dots,g_m)^T\in\R^m
\]
is said to have the \emph{standard real Gaussian distribution}, written
\[
g\sim N(0,I_m),
\]
if it has density
\[
(2\pi)^{-m/2}\exp\!\left(-\frac{\|x\|_2^2}{2}\right),
\qquad x\in\R^m.
\]
Equivalently, the coordinates $g_1,\dots,g_m$ are independent $N(0,1)$ random
variables.

More generally, if $\Sigma\in\R^{m\times m}$ is a symmetric positive
semidefinite matrix, then
\[
X\sim N(0,\Sigma)
\]
means that $X$ is an $m$-dimensional centered Gaussian random vector with
covariance matrix $\Sigma$, that is,
\[
\EE X=0,
\qquad
\operatorname{Cov}(X)=\EE[XX^T]=\Sigma.
\]
Equivalently, there exist a matrix $B\in\R^{m\times m}$ satisfying
\[
BB^T=\Sigma
\]
and a random vector $\eta\sim N(0,I_m)$ such that
\[
X=B\eta.
\]

Similarly, a random vector
\[
\zeta=(\zeta_1,\dots,\zeta_m)^T\in\C^m
\]
is said to have the \emph{standard complex Gaussian distribution}, written
\[
\zeta\sim CN(0,I_m),
\]
if it has density
\[
\pi^{-m}e^{-\|z\|_2^2},
\qquad z\in\C^m\cong\R^{2m}.
\]
Equivalently, the coordinates $\zeta_1,\dots,\zeta_m$ are independent
standard complex Gaussian random variables, each with scalar density
\[
\pi^{-1}e^{-|z|^2},
\qquad z\in\C.
\]

If $\Sigma\in\C^{m\times m}$ is Hermitian positive semidefinite, then
\[
Z\sim CN(0,\Sigma)
\]
means that $Z$ is an $m$-dimensional centered complex Gaussian random vector
with covariance matrix $\Sigma$, that is,
\[
\EE Z=0,
\qquad
\operatorname{Cov}(Z):=\EE[ZZ^*]=\Sigma.
\]
Equivalently, there exist a matrix $B\in\C^{m\times m}$ satisfying
\[
BB^*=\Sigma
\]
and a random vector $\xi\sim CN(0,I_m)$ such that
\[
Z=B\xi.
\]

In this paper, $CN(0,\Sigma)$ is always understood in this proper complex
Gaussian sense.

\subsection{Two standard distributional identities}

We first recall that if $\xi\sim CN(0,1)$, then $|\xi|^2$ has the exponential
distribution with rate $1$, that is, $\operatorname{Exp}(1)$ with density
\[
f(t)=e^{-t}\mathbf 1_{[0,\infty)}(t),
\]
where
$
\mathbf 1_{[0,\infty)}(t)=
\begin{cases}
	1,& t\ge 0,\\
	0,& t<0.
\end{cases}.
$
Indeed, for $t\ge 0$,
\[
\PP(|\xi|^2\le t)
=
\int_{|z|^2\le t}\pi^{-1}e^{-|z|^2}\,dz.
\]
Passing to polar coordinates gives
\[
\PP(|\xi|^2\le t)
=
\frac{1}{\pi}\int_0^{2\pi}\int_0^{\sqrt t} e^{-r^2}r\,dr\,d\theta
=
1-e^{-t}.
\]
Hence $|\xi|^2$ has density
\[
e^{-t}\mathbf 1_{[0,\infty)}(t).
\]

Next, if $Z_1$ and $Z_2$ are independent $N(0,1)$ random variables, then
$
Z_1/Z_2
$
has the standard Cauchy distribution; see~\cite[p.~168, 3.8.7(ii)]{Dur19}. In
particular, for every $s\ge 0$,
\[
\PP\!\left(\left|\frac{Z_1}{Z_2}\right|<s\right)
=
\int_{-s}^s \frac{dt}{\pi(1+t^2)}
=
\frac{2}{\pi}\arctan s.
\]

\subsection{Orthogonal and unitary invariance}

If $g\sim N(0,I_m)$ and $Q\in\R^{m\times m}$ is orthogonal, then
\[
Q^Tg\sim N(0,I_m).
\]
Indeed, the density of $g$ depends only on $\|x\|_2$, and is therefore
invariant under orthogonal transformations.

Similarly, if $\zeta\sim CN(0,I_m)$ and $Q\in\C^{m\times m}$ is unitary, then
\[
Q^*\zeta\sim CN(0,I_m).
\]
Again, this follows from the fact that the density
\[
\pi^{-m}e^{-\|z\|_2^2}
\]
depends only on the Euclidean norm of $z\in\C^m$.

\begin{proof}[Proof of Lemma~\ref{lem:gaussian-facts}]
	Part \textup{(a)} is exactly the orthogonal invariance of the standard real
	Gaussian distribution recorded above. Part \textup{(b)} is the corresponding
	unitary invariance in the complex setting. Part \textup{(c)} is the
	exponential identity for the squared modulus of a scalar complex Gaussian
	variable. Part \textup{(d)} is the classical Cauchy-ratio fact, and the
	displayed probability formula follows by integrating the Cauchy density over
	$[-s,s]$:
	\[
	\PP\!\left(\left|\frac{Z_1}{Z_2}\right|<s\right)
	=
	\int_{-s}^s \frac{dt}{\pi(1+t^2)}
	=
	\frac{2}{\pi}\arctan s.
	\qedhere
	\]
\end{proof}
\end{document}